\documentclass[a4paper]{article}
\usepackage[round]{natbib}

\usepackage{amsthm,amsmath}
\usepackage{amsmath, amsfonts, amsthm, amssymb, bbm, bm}
\usepackage{mathrsfs}
\usepackage{enumerate}
\usepackage{tabularx}
\usepackage{multicol}
\usepackage{multirow}
\usepackage{color}
\usepackage{tikz}
\usetikzlibrary{shapes, calc, arrows}

\usepackage{graphicx}
\usepackage{comment}
\usepackage{stmaryrd}

\RequirePackage[colorlinks,citecolor=blue,urlcolor=blue]{hyperref}

\newcommand{\reals}{\mathbb{R}}
\newcommand{\G}{\mathcal{G}}
\newcommand{\X}{\mathfrak{X}}
\newcommand{\bs}{\boldsymbol}
\newcommand{\Ind}{\mathbbm{1}}
\newcommand{\vep}{\varepsilon}

\DeclareMathOperator{\pa}{pa}
\DeclareMathOperator{\ch}{ch}

\DeclareMathOperator{\dec}{de}

\DeclareMathOperator{\sterile}{sterile}

\theoremstyle{plain}
\newtheorem{lem}{Lemma}[section]
\newtheorem{thm}[lem]{Theorem}
\newtheorem{prop}[lem]{Proposition}
\newtheorem{cor}[lem]{Corollary}

\theoremstyle{definition}
\newtheorem{dfn}[lem]{Definition}

\newtheorem{exm}[lem]{Example}

\newcommand\indep{\protect\mathpalette{\protect\independenT}{\perp}}
\def\independenT#1#2{\mathrel{\rlap{$#1#2$}\mkern2mu{#1#2}}}

\newcommand{\M}{{\mathcal{M}}}

\renewcommand{\X}{{\mathfrak{X}}}

\newcommand{\F}{\mathcal{F}}
\renewcommand{\P}{{p}}
\newcommand{\p}{p}
\newcommand{\x}{{\bs x}}

\newcommand{\TC}{{\rm TC}}
\newcommand{\TS}{{\rm TS}}

\begin{document}




%

\title{Margins of discrete Bayesian networks}
\author{Robin J.\ Evans}
\maketitle

\begin{abstract}
Bayesian network models with latent variables are widely used in statistics 
and machine learning.  In this paper we provide a complete algebraic 
characterization of Bayesian network models with latent variables 
when the observed variables are discrete and no assumption is made
about the state-space of the latent variables.  We show that it is 
algebraically equivalent 
to the so-called nested Markov model, meaning that the two are the
same up to inequality constraints on the joint probabilities. In 
particular these two models have the same dimension.  The nested Markov model 
is therefore the best possible description of the latent variable 
model that avoids consideration of inequalities, 
which are extremely complicated in general.  A consequence
of this is that the constraint finding algorithm of 
\citet{tian:02} is complete for finding equality constraints.

Latent variable models suffer
from difficulties of unidentifiable parameters and non-regular
asymptotics; in contrast the nested Markov model is fully
identifiable, represents a curved exponential family of known
dimension, and can easily be fitted using an explicit
parameterization.  
\end{abstract}

%
%

\section{Introduction}

Directed acyclic graph (DAG) models, also known as Bayesian network models,
are widely used multivariate models in probabilistic reasoning,
machine learning and causal inference \citep{bishop:07, darwiche:09,
  pearl:09}.  These models are defined by simple factorizations of the
joint distribution, and in the case of discrete or jointly Gaussian
random variables, are curved exponential families of known dimension.
The inclusion of latent variables within Bayesian network models can
greatly increase their flexibility, and also account for unobserved
confounding.  However, this flexibility comes at the cost of creating
models that are not easy to explicitly describe when considered as
marginal models over the observed variables.  Latent variable models
generally do not have fully identifiable parameterizations, and
contain `singularities' that lead to non-regular asymptotics
\citep{drton:09}.  In addition, using them may force a modeller to
specify a parametric structure over the latent variables, introducing
additional assumptions that are generally difficult to test and may be
unreasonable.

If no parametric assumptions are made about the latent variables, and
no assumption is made about their state-space, this leads to an
implicitly defined \emph{marginal model}.  The marginal model has the
advantage of avoiding some of the assumptions made by a parametric
latent variable model, however no explicit characterization of the
model is available, and nor is there any obvious method for fitting it
to data.

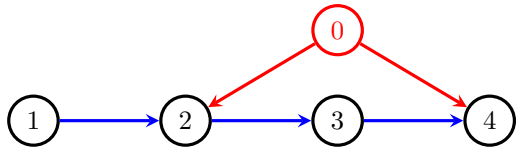
\begin{figure}
 \begin{center}
 \begin{tikzpicture}
 [rv/.style={circle, draw, very thick, minimum size=6.5mm, inner sep=0.75mm}, node distance=20mm, >=stealth]
 \pgfsetarrows{latex-latex};
 \node[rv]  (1)              {$1$};
 \node[rv, right of=1] (2) {$2$};
 \node[rv, right of=2] (3) {$3$};
 \node[rv, right of=3] (4) {$4$};
 \node at (3) [rv, color=red, yshift=12mm] (U) {$0$};
 \draw[->, very thick, color=blue] (1) -- (2);
 \draw[->, very thick, color=blue] (2) -- (3);
 \draw[->, very thick, color=blue] (3) -- (4);
 \draw[<-, very thick, color=red] (2) -- (U);
 \draw[->, very thick, color=red] (U) -- (4);
 \end{tikzpicture}
 \caption{A directed acyclic graph on five vertices.}
 \label{fig:dag}
 \end{center}
\end{figure}

\begin{exm} \label{exm:one} Consider the DAG on five vertices shown in
  Figure \ref{fig:dag}.  The graph represents a multivariate model
  over five random variables $X_0$, $X_1$, $X_2$, $X_3$ and $X_4$, with the
  restriction that the joint density factorizes as
\begin{align*}
\p(x_0, x_1, x_2, x_3, x_4)
& = \p(x_0) \cdot \p(x_1) \cdot \p(x_2 \,|\, x_0, x_1) \cdot \p(x_3 \,|\, x_2) \cdot \p(x_4 \,|\, x_0, x_3);
\end{align*}
here, for example, $p(x_3 \,|\, x_2)$ represents the conditional
density of $X_3$ given $X_2$.  This model arises naturally in the
context of dynamic treatment regimes and longitudinal exposures
\citep{robins:86}: $X_1$ and $X_3$ represent treatments and 
$X_2$ and $X_4$ some outcome of interest.  The treatments are
randomized, though the second treatment $X_3$ may depend upon 
the first outcome $X_2$, for example a dose may be dynamically 
adjusted.  Since the outcomes are measured on the same patient,
they are assumed to be correlated due to a common cause $X_0$,
which might represent an underlying health status, as well as genetic 
and lifestyle factors.

If we treat $X_0$ as a latent variable, the
\emph{marginal model} over the remaining observed variables
$(X_1, X_2, X_3, X_4)$ is the collection of probability distributions
that can be written in the form
\begin{align}
\lefteqn{\p(x_1, x_2, x_3, x_4)} \nonumber\\
& \qquad = \int_{\X_0} \p(x_0) \cdot \p(x_1) \cdot \p(x_2 \,|\, x_0, x_1) \cdot \p(x_3 \,|\, x_2) \cdot \p(x_4 \,|\, x_0, x_3) \, dx_0. \label{eqn:model}
\end{align}
That is, any $(X_1,X_2,X_3,X_4)$-margin of a distribution which
factorizes according to the DAG over all five variables, for any
state-space or distribution of $X_0$\footnote{In general
 it is sufficient to assume hidden variables are uniform on $(0,1)$; 
	for this particular graph, we can choose $X_0$ to be
  finite and discrete without loss of generality provided 
  it has a sufficiently large number of states.}.

From (\ref{eqn:model}) we can deduce that the conditional independence
$X_3 \indep X_1 \,|\, X_2$ holds in the marginal model; i.e.\
\begin{align}
p(x_3 \,|\, x_1, x_2) = p(x_3 \,|\, x_2). \label{eqn:ci}
\end{align}
In addition this model satisfies the so-called \emph{Verma constraint}
of \citet{robins:86} \citep[see also][]{verma:90}, because the
expression
\begin{align}
q(x_4 \,|\, x_3) \equiv \sum_{x_2} p(x_2 \,|\, x_1) \cdot p(x_4 \,|\, x_1, x_2, x_3) \label{eqn:vc}
\end{align}
does not depend upon $x_1$ (see Example \ref{exm:verma}).  

The set of distributions satisfying both (\ref{eqn:ci}) and
(\ref{eqn:vc}) is a so-called \emph{nested Markov model}
\citep{richardson:17}.  If the four observed variables are binary these
equations represent four independent constraints, and the nested model is
therefore an 11-dimensional subset of the 15-dimensional probability
simplex.  

It is not immediately clear whether or not this nested model is the
same as the marginal model defined by (\ref{eqn:model}): in 
principle the marginal model might impose additional restrictions
beyond (\ref{eqn:ci}) and (\ref{eqn:vc}).  
This begs the question, is the set of distributions that satisfy 
(\ref{eqn:model}) characterized by (\ref{eqn:ci}) and (\ref{eqn:vc})?  
\end{exm}

The answer turns out to be `almost', in the sense that the set of
distributions that can be written in the form (\ref{eqn:model}) is a full-dimensional
subset of the set that satisfy (\ref{eqn:ci}) and (\ref{eqn:vc}), though 
there are additional inequality constraints.  
This situation is represented by Figure
\ref{fig:diagram}, which shows the marginal model ($\M$, in blue) 
lying strictly within the nested model ($\mathcal{N}$, in red), 
but the two having the same dimension.  

This paper shows that 
this result holds generally for all models of this kind.  That is,
the constraints on the model are precisely those derived from the 
algorithm of \citet{tian:02}, as represented by the statistical 
model of \citet{richardson:17}.

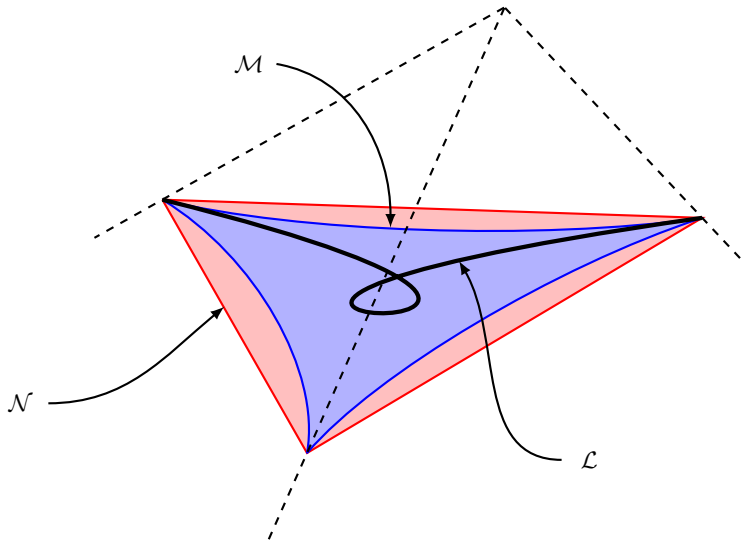
\begin{figure}
\begin{center}
\begin{tikzpicture}[scale=1.5,every node/.style={minimum size=0.7cm}]
\small
\begin{scope}
[
 every node/.append style={yslant=0.6,xslant=-1},yslant=0.6,xslant=-1
   ]
  \node[minimum size=0mm, inner sep=0mm] (v0) at (90:2) {};
  \node[minimum size=0mm, inner sep=0mm] (v1) at (210:2) {};
  \node[minimum size=0mm, inner sep=0mm] (v2) at (330:2) {};
  \draw[thick, red, fill=red!25] (v0.center) -- (v1.center) -- (v2.center) -- (v0.center);
\draw [blue, thick, fill=blue!30] (v0.260) .. controls +(260:1) and +(40:1) .. (v1.40)
-- (v1.20) .. controls +(20:1) and +(160:1) .. (v2.160)
-- (v2.140) .. controls +(140:1) and +(280:1) .. (v0.280);
\draw[ultra thick] (v0.center) .. controls +(280:5) and +(140:5) .. (v2.center);
\coordinate (s) at (2,1);
\coordinate (m) at (45:0.78);
\coordinate (n) at (135:1.03);
\coordinate (l) at (0.5,-0.1);
\end{scope}

\draw[-latex,thick](-1,2) node[left,scale=1]{$\mathcal{M}$}
  to[out=350,in=90] (m);
\draw[-latex,thick](-3,-1) node[left,scale=1]{$\mathcal{N}$}
  to[out=0,in=220] (n);  
\draw[-latex,thick](1.5,-1.5) node[right,scale=1, align=center]{$\mathcal{L}$}
  to[out=180,in=300] (l);
  
\node[minimum size=0mm, inner sep=0mm] (v3) at (1,2.5) {};
\node[minimum size=0mm, inner sep=0mm] (w0) at ($1.2*(v0)-.2*(v3)$) {};
\node[minimum size=0mm, inner sep=0mm] (w1) at ($1.2*(v1)-.2*(v3)$) {};
\node[minimum size=0mm, inner sep=0mm] (w2) at ($1.2*(v2)-.2*(v3)$) {};

\draw[black,thick,dashed] (v3) -- (w0);
\draw[black,thick,dashed] (v3) -- (w1);
\draw[black,thick,dashed] (v3) -- (w2);
\normalsize
\end{tikzpicture}
\end{center}
\caption{Diagrammatic representation of the probability simplex (dashed outline) 
and a marginal model ($\M$, in blue) sitting strictly within the associated
nested model ($\mathcal{N}$, in red); note the two models have the same dimension.  
  Any parametric latent variable model will be contained strictly within 
  $\M$, but it may have a smaller dimension and be non-regular 
  (an example is shown as $\mathcal{L}$).  The `ordinary Markov model' is
  not shown, but contains $\mathcal{N}$ and would generally have 
  larger dimension.}
\label{fig:diagram}
\end{figure}

Existing approaches to the problem of describing Bayesian 
network models with hidden variables either make use of parametric
structure on the latent variables \citep[for example,][]{silva:09,
  anandkumar:13}, or are restricted to testing conditional
independences and do not consider constraints such as (\ref{eqn:vc}). 
This latter category includes the ancestral graph models of
\citet{richardson:02} and the equivalent\footnote{The models are
  equivalent if selection variables are not present, which is the case
  throughout this paper.} models on acyclic directed mixed graphs
(ADMGs) of \citet{richardson:03}; these pure conditional
independence models, which we refer to as the \emph{ordinary Markov
  models}, generally have a larger dimension than any latent variable
model, so using them as a proxy leads to a loss of power to distinguish
between certain kinds of model.  

On the other hand parametric hidden variable models suffer from various problems
caused by the choice of state-space.  They may be `too large', in 
the sense that the dimension of the full model is greater than the 
dimension of the model over the observed data, thereby introducing
identifiability problems.  They may also be `too small', in 
that unwanted additional restrictions are implied by the parametric 
structure, and therefore the models have a smaller dimension
than the marginal model: this is depicted by the curve labelled
$\mathcal{L}$ in Figure \ref{fig:diagram}. 

Paradoxically, it may even be the case that a hidden variable
model is `too large' and `too small' at the same time!  For example,
if we use a latent variable in Example \ref{exm:one} with the simplest
possible state-space in which everything is binary, then the full 
model over all five variables has dimension 12; however, we have
already established that the dimension of the marginal model over the observed 
variables is at most 11, so the model is clearly over-parameterized.  
In fact, it can be shown that the 
dimension of this latent variable model over the observed variables is 
only 10, so an additional artificial restriction is present due to the 
choice of a binary latent variable model
(see Appendix \ref{sec:deg}).  

If $X_0$ is given enough states, the latent variable 
model and the marginal model coincide for graphs such 
as the one in Figure \ref{fig:dag}, a fact we will 
exploit in our proofs.  However, this latent variable 
model is less useful for statistical inference because
it is generally massively over-parameterized. 

\subsection{A Short Algebra Tutorial}

This paper makes use of some results from real algebraic
geometry, which provides powerful tools for analysing these 
complicated sets of distributions.  All our statistical 
models are collections of distributions 
within the probability simplex that satisfy certain 
constraints.  
The constraints on a Bayesian network model are conditional
independences, and can be 
represented as the requirement that certain polynomials in the 
probabilities are equal to zero; for 
example the conditional independence $X_1 \indep X_3 \mid X_2$ is
equivalent to
\begin{align*}
p(x_2) \cdot p(x_1, x_2, x_3) - p(x_1, x_2) \cdot p(x_2, x_3) = 0 \qquad \forall x_1, x_2, x_3.
\end{align*}
A set defined by the zeros of polynomials is said to be an
\emph{algebraic variety}, or sometimes an \emph{algebraic set}.  
In addition to equality constraints, these models will satisfy polynomial inequalities;
i.e.\ $p(x_V) \geq 0$.  A set defined by a combination of
polynomial equalities and inequalities is said to be \emph{semi-algebraic};
this category includes many common finite-dimensional statistical models. 
Semi-algebraic sets have the nice property that when we eliminate 
one of the variables or project
onto a linear subspace, they remain semi-algebraic, generally known as the 
Tarski-Seidenberg theorem \citep[see, for example,][Theorem 2.72]{basu:06}.   A consequence
of this is that the margin of any model defined by a semi-algebraic set 
is also defined by a semi-algebraic set. 

The \emph{Zariski closure} of a set is the smallest algebraic variety
that contains it; the fact that this is well-defined is a significant 
result in algebraic geometry.  For a semi-algebraic set one can informally think of
its Zariski closure as the set obtained by keeping the equality constraints 
and `throwing away' the inequality constraints.  

Semi-algebraic sets have many interesting properties, but they are not 
necessarily `nice' from a statistical perspective, in the sense of 
leading to regular asymptotics.  For this we need our set to be a 
\emph{manifold}, i.e.\ to be locally Euclidean.



\subsection{Contribution}

In this paper we show that marginal models with finite discrete
observed variables are algebraically equivalent to the appropriate
nested Markov model, in the sense that the Zariski closures of the
marginal model and the nested model are the same.  
A consequence of this is that a margin of a DAG model and its nested
counterpart have the same dimension, and differ only by inequality
constraints.  The marginal model defined by (\ref{eqn:model}) 
in Example \ref{exm:one} is indeed
11-dimensional, and is algebraically defined by (\ref{eqn:ci}) and
(\ref{eqn:vc}); however, the marginal model also satisfies polynomial
inequality constraints that the nested model does not.  
The result can be interpreted as showing that the constraint
finding algorithm of \citet{tian:02} is `complete', in the sense that
there are no other equality constraints to find without making 
further assumptions.

This means that we have, for the first time, a full algebraic
characterization of margins of Bayesian network models.  It also shows
that the nested model represents a sensible and pragmatic
approximation to the marginal model: inequality constraints are
typically extremely complicated, so the nested model---which has
a factorization criterion, separation criteria, and a discrete
parameterization \citep{richardson:17}---is much easier to work with, and can easily be
fitted with existing algorithms \citep{evans:10}.  In addition, the
nested model inside the probability simplex is a manifold and therefore 
regular whenever the joint distribution is positive,
whereas the marginal model may have a boundary that lies strictly
inside the simplex.  The nested model therefore has better statistical 
properties than the marginal model, in the sense that data generated from
any strictly positive distribution will lead to regular asymptotics.  

Causal discovery methods such as the FCI algorithm that use conditional 
independence constraints could, in principle, be extended to the 
constraints implied by nested models
\citep{sgs:00}; our main result shows that is `as good as it gets',
in the sense that there are no other equality constraints to test 
without making further (e.g.\ parametric) assumptions.  Thus, this paper
probes the limits of what it is possible to learn about causal models
with hidden variables from observational data. 

We work with a class of hyper-graphs called mDAGs, with which we
associate marginals of DAG models \citep{evans:mdag}.  The remainder
of the paper is organized as follows: Section \ref{sec:dags}
introduces DAG models, their margins and mDAGs, and carefully defines
the problem of interest.  Section \ref{sec:nested} describes the
nested Markov property. Section \ref{sec:geared} describes latent variable 
models that can be used to represent the marginal model without loss of
generality, and Section \ref{sec:main} contains the main results of
the paper.  Finally, in Section \ref{sec:smooth} we show that a large
class of marginal models represent smooth manifolds, and provide some
discussion.

\section{Directed Graphical Models} \label{sec:dags}

We begin with some elementary graphical definitions.

\begin{dfn}
  A \emph{directed graph}, $\G(V, \mathcal{E})$, consists of a 
  finite set of vertices, $V$, and a collection of edges, 
  $\mathcal{E}$, which are ordered pairs of distinct
  elements of $V$.  If $(v, w) \in \mathcal{E}$ we denote this by $v \rightarrow
  w$, and say that $v$ is a \emph{parent} of $w$; the set of parents
  of $w$ is denoted by $\pa_\G(w)$.  Similarly $w$ is a \emph{child} of
  $v$, and the child set is denoted by $\ch_\G(v)$.

  A directed graph is \emph{acyclic} if there is no sequence of
  edges $v_1 \rightarrow v_2 \rightarrow \cdots \rightarrow v_k
  \rightarrow v_1$ for $k > 1$.  We call such a graph a \emph{directed
    acyclic graph}, or DAG.
\end{dfn}

Graphs are best understood visually: an example of a DAG with five
vertices and five edges is given in Figure \ref{fig:dag}.  We will
require the following generalization of a DAG that allows for two
separate types of vertex.

\begin{dfn}
  A \emph{conditional DAG} $\G(V, W, \mathcal{E})$ is a DAG with
  vertices $V \dot\cup W$\footnote{Here and throughout $\dot\cup$
    denotes a disjoint union of sets.} and edge set $\mathcal{E}$,
  with the restriction that no vertex in $W$ may have any parents.
  The vertices in $V$ are the \emph{random vertices}, and $W$ the
  \emph{fixed vertices}; these two sets are disjoint.
\end{dfn}

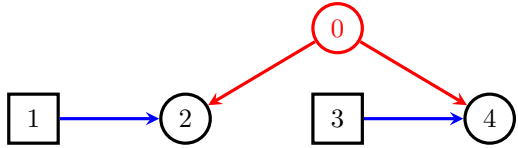
\begin{figure}
 \begin{center}
 \begin{tikzpicture}
 [rv/.style={circle, draw, very thick, minimum size=6.5mm, inner sep=0.75mm}, node distance=20mm, >=stealth]
 \pgfsetarrows{latex-latex};
 \node[rv, rectangle]  (1)              {$1$};
 \node[rv, right of=1] (2) {$2$};
 \node[rv, rectangle, right of=2] (3) {$3$};
 \node[rv, right of=3] (4) {$4$};
 \node at (3) [rv, color=red, yshift=12mm] (U) {$0$};
 \draw[->, very thick, color=blue] (1) -- (2);
 \draw[->, very thick, color=blue] (3) -- (4);
 \draw[<-, very thick, color=red] (2) -- (U);
 \draw[->, very thick, color=red] (U) -- (4);
 \end{tikzpicture}
 \caption{A conditional directed acyclic graph with three random 
 vertices ($0,2,4$) and two fixed vertices ($1,3$).}
 \label{fig:dag2}
 \end{center}
\end{figure}

If $W = \emptyset$, this reduces to the ordinary definition of a DAG.
We denote fixed vertices with square nodes, and random ones with round
nodes: see the example in Figure \ref{fig:dag2}.  

\subsection{Graphical Models}

A graphical model arises from the identification of a graph with a
collection of multivariate probability distributions; see
\citet{lau:96} for an introduction.  Each vertex $v \in V$
represents a random variable $X_v$ taking values in a finite
state-space $\X_v$, and a model for their joint distribution
is determined by the structure of the graph.
With a conditional DAG $\G$ we associate a 
collection of probability measures $P(\cdot \mid x_W)$ on $\X_V \equiv \times_{v \in V}
\X_v$, indexed by $x_W \in \X_W$.  
Mathematically, fixed nodes play a similar role to 
the `parameter nodes' used by \citet{dawid:02}.

Following \citet{lau:96}, we say a \emph{probability kernel} over $\X_A$ given $\X_B$ 
is a non-negative function $q : \X_A \times \X_B \rightarrow \reals$ 
such that $\sum_{x_A} q(x_A \mid x_B) = 1$ for all $x_B \in \X_B$.
A kernel behaves much like a conditional probability distribution, 
but no assumption is made about any distribution over the indexing
set $\X_B$.

We apply the usual definitions for marginalizing and conditioning
in kernels:
\begin{align*}
q(x_A \mid x_B) &\equiv \sum_{x_C} q(x_A, x_C \mid x_B), && q(x_A \mid x_B, x_C) \equiv \frac{q(x_A, x_C \mid x_B)}{q(x_C \mid x_B)}.
\end{align*}
If $q(x_A \mid x_B, x_C)$ does not depend upon $x_B$ then we will
denote it $q(x_A \mid x_C)$, and say that $X_A \indep X_B \mid X_C \, [q]$. 
Here, and elsewhere, we use the shorthand ${VW}$ for $V \cup W$ in 
subscripts.

\begin{dfn}
  Let $\p(x_V \mid x_W)$ be a probability kernel over 
  $\X_V$ indexed by $\X_W$.  We say that $\p$ obeys the
  \emph{factorization criterion} with respect to a DAG $\G$ if it
  factorizes into univariate kernels as
\begin{align}
\p(x_V \,|\, x_W) = \prod_{v \in V} \p(x_v \,|\, x_{\pa(v)}), \qquad x_{VW} \in \X_{VW}. \label{eqn:markov}
\end{align}
\end{dfn}

The definition reduces to the familiar factorization criterion for
DAGs if $W = \emptyset$.  The extra generality will be useful for
discussing Markov properties which involve factorization of the
distribution into conditional pieces.  The fixed vertices are 
analogous to variables that have been conditioned upon; if $p$ satisfies
(\ref{eqn:markov}) then, after renormalization, it also
satisfies the factorization criterion for the same DAG with all
vertices random.

The definition of a Bayesian network can be extended to the case where
no joint density exists by insisting that each random variable $X_v$
can be written as a measurable function of $X_{\pa(v)}$ and an
independent noise variable; we call this the \emph{structural equation
  property}.  If the density exists the two criteria are equivalent,
and since we work with discrete variables this condition is always
satisfied.  Although the factorization property is often simpler to
work with for practical purposes such as modelling and fitting, the
structural equation property is useful in proofs.  The well-known
global Markov property based on d-separation is also equivalent to the
structural equation property \citep{pearl:09}.

\begin{exm} \label{exm:ci} A distribution $P$ with density $\p$ obeys
  the factorization criterion for the graph in Figure \ref{fig:dag} if
  the density has the form
\begin{align*}
p(x_0, x_1, x_2, x_3, x_4)
&= \p(x_0) \cdot \p(x_1) \cdot \p(x_2 \,|\, x_0, x_1) \cdot \p(x_3 \,|\, x_2) \cdot \p(x_4 \,|\, x_0, x_3).
\end{align*}
Such distributions are precisely those which satisfy the conditional
independences
\begin{align*}
X_1 \indep X_0, \qquad X_3 \indep X_0, X_1 \,|\, X_2, \qquad X_4 \indep X_{1}, X_2 \,|\, X_0,X_3. 
\end{align*}
\end{exm}

\begin{exm}
A kernel $\p$ obeys the factorization criterion for the 
conditional DAG in Figure \ref{fig:dag2} if it can be written as
\begin{align*}
p(x_0, x_2, x_4 \,|\, x_1, x_3) = \p(x_0)\cdot \p(x_2 \,|\, x_0, x_1) \cdot \p(x_4 \,|\, x_0, x_3).
\end{align*}
\end{exm}

\subsection{Latent Variables and mDAGs} \label{sec:mdags}

We now introduce the possibility that some of the random variables are
unobserved or \emph{latent}, leaving the marginal distribution over
the remaining \emph{observed} variables.  
We represent the collection of margins of DAG models
using a larger class of hyper-graphs called mDAGs (`marginal
DAGs').  These avoid dealing with latent variables directly, instead
introducing additional edges to represent them.  For example, the DAG
in Figure \ref{fig:dag}, with the vertex 0 treated as a latent
variable, is represented by the mDAG in Figure \ref{fig:verma:mdag}.




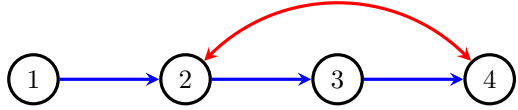
\begin{figure}
 \begin{center}
 \begin{tikzpicture}
 [rv/.style={circle, draw, very thick, minimum size=6.5mm, inner sep=0.75mm}, node distance=20mm, >=stealth]
 \pgfsetarrows{latex-latex};
\begin{scope}
 \node[rv]  (1)              {$1$};
 \node[rv, right of=1] (2) {$2$};
 \node[rv, right of=2] (3) {$3$};
 \node[rv, right of=3] (4) {$4$};
 \draw[->, very thick, color=blue] (1) -- (2);
 \draw[->, very thick, color=blue] (2) -- (3);
 \draw[->, very thick, color=blue] (3) -- (4);
 \draw[<->, very thick, color=red] (2.45) .. controls +(1,1) and +(-1,1) .. (4.135);
\end{scope}
 \end{tikzpicture}
 \caption{An mDAG representing the DAG in Figure \ref{fig:dag}, with the vertex 0 treated as unobserved.}
 \label{fig:verma:mdag}
 \end{center}
\end{figure}

Define an \emph{abstract simplicial complex} $\mathcal{B}$ over $V$ as
a collection of non-empty subsets of $V$ such that (i)
$\{v\} \in \mathcal{B}$ for every $v \in V$, and (ii) if
$A \in \mathcal{B}$ and $B \subseteq A$ with $B \neq \emptyset$, then $B \in \mathcal{B}$.

\begin{dfn}
  An \emph{mDAG}, $\G(V, W, \mathcal{E}, \mathcal{B})$, is hyper-graph
  consisting of a conditional DAG with random vertices $V$, fixed
  vertices $W$ and directed edge set $\mathcal{E}$, together with an
  abstract simplicial complex $\mathcal{B}$ over $V$, called the
  \emph{bidirected faces}.
  
  We say that $\G'(V', W', \mathcal{E}', \mathcal{B}')$ is a 
  \emph{subgraph} of $\G$ if $V' \subseteq V$, 
  $\mathcal{E}' \subseteq \mathcal{E}$,   $\mathcal{B}' \subseteq \mathcal{B}$, 
  and $W' \subseteq V \cup W$: that is, each component is contained 
  within the previous one, but random vertices may become
  fixed.
\end{dfn}

The mDAG was introduced by \citet{evans:mdag}, without the additional
generality of fixed vertices.  This aspect changes very little to the
theory of these graphs, but is necessary for understanding the nested
Markov model; note that bidirected faces only involve the random
vertices.  As with conditional DAGs, when representing mDAGs
graphically the fixed vertices are drawn as square nodes and random
vertices as circles.  

The bidirected simplicial complex is represented
by its maximal non-trivial elements (i.e.\ those of size at least 2), 
called the \emph{bidirected hyperedges}, or just \emph{edges}.  
These are drawn in red, as in
Figure \ref{fig:mdag}(a); in this case $W = \{6\}$ and the maximal
sets of $\mathcal{B}$ are $\{1,2\}$, $\{2,3,4\}$, and $\{3,4,5\}$.

\begin{figure}
\begin{center}
\begin{tikzpicture}[rv/.style={circle, draw, very thick, minimum size=6.5mm, inner sep=1mm}, node distance=25mm, >=stealth]
 \pgfsetarrows{latex-latex};
\begin{scope}
 \node[circle, minimum size=2mm, inner sep=0mm, fill=red] (U2) {};
 \node[rv] (2) at (150:1.7) {$2$};
 \node[rv] (4) at (270:1.7) {$4$};
 \node[rv] (3) at (30:1.7) {$3$};
 \node[circle, minimum size=2mm, inner sep=0mm, fill=red, yshift=-17mm] (U3) at (3) {};
 \node[rv] (5) at ($(U3)+(330:1.7)$) {$5$};
 \node[rv, left of=2] (1) {$1$};
 \node[rv, left of=4, xshift=5mm, rectangle] (6) {$6$};
 \draw[<->, very thick, color=red] (1) -- (2);
 \draw[->, very thick, color=red] (U2) -- (2);
 \draw[->, very thick, color=red] (U2) -- (3);
 \draw[->, very thick, color=red] (U2) -- (4);
 \draw[->, very thick, color=red] (U3) -- (5);
 \draw[->, very thick, color=red] (U3) -- (3);
 \draw[->, very thick, color=red] (U3) -- (4);
 \draw[->, very thick, color=blue] (2) -- (4);
 \draw[->, very thick, color=blue] (3) -- (5);
 \draw[->, very thick, color=blue] (6) -- (4);
 \draw[->, very thick, color=blue] (3) -- (4);
 \draw[->, very thick, color=blue] (1) .. controls +(1.5,1.5) and +(-1.5,1.5) .. (3);
\node[below of=6, yshift=17mm, xshift=5mm] {(a)};
\end{scope}
\begin{scope}[yshift=-5.5cm]
 \node[rv, color=red] (U2) {$u_2$};
 \node[rv] (2) at (150:1.7) {$2$};
 \node[rv] (4) at (270:1.7) {$4$};
 \node[rv] (3) at (30:1.7) {$3$};
 \node[rv, left of=2] (1) {$1$};
 \node[rv, left of=2, xshift=12.5mm, color=red] (U1) {$u_1$};
 \node[rv, yshift=-17mm, color=red] (U3) at (3) {$u_3$};
 \node[rv] (5) at ($(U3)+(330:1.7)$) {$5$};
 \node[rv, left of=4, xshift=5mm, rectangle] (6) {$6$};
 \draw[->, very thick, color=red] (U1) -- (2);
 \draw[->, very thick, color=red] (U1) -- (1);
 \draw[->, very thick, color=red] (U2) -- (2);
 \draw[->, very thick, color=red] (U2) -- (3);
 \draw[->, very thick, color=red] (U2) -- (4);
 \draw[->, very thick, color=red] (U3) -- (5);
 \draw[->, very thick, color=red] (U3) -- (3);
 \draw[->, very thick, color=red] (U3) -- (4);
 \draw[->, very thick, color=blue] (2) -- (4);
 \draw[->, very thick, color=blue] (6) -- (4);
 \draw[->, very thick, color=blue] (3) -- (4);
 \draw[->, very thick, color=blue] (3) -- (5);
 \draw[->, very thick, color=blue] (1) .. controls +(1.5,1.5) and +(-1.5,1.5) .. (3);
\node[below of=6, yshift=17mm, xshift=5mm] {(b)};
\end{scope}
\end{tikzpicture}
\caption{(a) An mDAG, $\G$, and (b) a DAG with hidden variables,
  $\bar{\G}$, representing the same model (the canonical DAG).}
 \label{fig:mdag}
\end{center}
\end{figure}

With each mDAG, $\G$, we can associate a conditional DAG $\bar{\G}$ by
replacing each maximal element $B \in \mathcal{B}$ (of size at least
2) with a new random vertex $u$, such that the children of $u$ are
precisely the vertices in $B$.  The new vertex $u$ becomes the
`unobserved' variable represented by the bidirected edge $B$.  We call
$\bar{\G}$ the \emph{canonical DAG} associated with $\G$.  The mDAG in
Figure \ref{fig:mdag}(a) is thus associated with the canonical DAG in
Figure \ref{fig:mdag}(b).

Our interest in mDAGs lies in their representation of the margin of
the associated canonical DAG, and so we define our model in this
spirit; see \citet{evans:mdag}.

\begin{dfn}
  Let $\G$ be an mDAG with vertices $V \dot\cup W$, and let $\bar{\G}$
  be the canonical DAG with vertices $V \dot\cup U \dot\cup W$.  A
  kernel $p$ over $\X_V$ indexed by $\X_W$ is said to be in the
  \emph{marginal model} for $\G$ if there exists a kernel $q$ that
  factorizes according to $\bar{\G}$, and
 \begin{align*}
 p(x_V \,|\, x_W) = \int_{\X_U} q(x_V, \, x_U \,|\, x_W) \, dx_U.
 \end{align*}
 That is, the margin of $q$ over $X_V$ is $p$.
    Denote the collection of such kernels by $\M(\G)$.
\end{dfn}

In other words, the marginal model is the collection of kernels
that could be constructed as the margin of a Bayesian network with
latent variables replacing the bidirected edges.  
If $\G$ is a DAG then the marginal model is just the
usual model defined by the factorization.

A latent variable model corresponding to a canonical DAG $\bar{\G}$
(i.e.\ possibly with parametric or distributional assumptions on the 
latent variables) always lies within the marginal model corresponding to 
the mDAG $\G$.  
This should not be taken as meaning that the marginal model supersedes 
all latent variable models, since sometimes the additional parametric
assumptions made in a latent variable model are crucial to their 
utility.  For example, representing hidden Markov models and 
phylogenetic tree models using an mDAG leads to a bidirected
hyper-edge containing all vertices; the marginal models are therefore 
saturated, and rather uninteresting from the perspective statistical inference. 

From the definitions above it may seem as though the set of marginal
DAG models that can be represented by mDAGs is restricted to cases
where the latent variables have no parents; in fact this does not
cause any loss of generality, since all marginal DAG models can
be represented in this way \citep[see][]{evans:mdag}.


\subsection{Districts and Sterile Vertices}

\begin{dfn}
  A collection of random vertices $C \subseteq V$ in an mDAG $\G$ is
  \emph{bidirected-connected} if for any distinct $v,w \in C$, there
  is a sequence of vertices $v = v_0, v_1, \ldots, v_k = w$ all in $C$
  such that, for each $i=1,\ldots,k$, the pair $\{v_{i-1}, v_{i}\} \in
  \mathcal{B}$.

  A \emph{district} of an mDAG is an inclusion maximal
  bidirected-connected set of random vertices.
\end{dfn}

More informally, a district is a maximal set of random vertices joined by the
red edges in an mDAG.  It is easy to see from the definition that
districts form a partition of the random vertices in an mDAG.  The
mDAG in Figure \ref{fig:verma:mdag}, for example, contains three
districts, $\{1\}$, $\{3\}$ and $\{2,4\}$.  Districts inspire a useful
reduction of mDAGs, via the following special subgraph.

\begin{dfn}
Let $\G$ be an mDAG containing random vertices $C \subseteq V$.  Then $\G[C]$ is the subgraph of $\G$ with
\begin{enumerate}[(i)]
\item random vertices $C$ and fixed vertices $\pa_\G(C) \setminus C$;
\item those directed edges $w \rightarrow v$ such that $v \in C$ (and
  $w \in \pa_\G(C)$);
\item the bidirected simplicial complex $\mathcal{B}_C \equiv \{B
  \cap C : B \in \mathcal{B}(\G)\}$.  
\end{enumerate}
\end{dfn}

\begin{figure}
 \begin{center}
 \begin{tikzpicture}
 [rv/.style={circle, draw, very thick, minimum size=6.5mm, inner sep=0.75mm}, node distance=20mm, >=stealth]
 \pgfsetarrows{latex-latex};
\begin{scope}
 \node[rv]  (1)              {$1$};
 \node[below of=1, yshift=10mm] {(a)};
\end{scope}
\begin{scope}[xshift=4cm]
 \node[rv, rectangle]  (2)              {$2$};
 \node[rv, right of=2] (3) {$3$};
 \draw[->, very thick, color=blue] (2) -- (3);
 \node[below of=2, xshift=10mm, yshift=10mm] {(b)};
\end{scope}
\begin{scope}[xshift=1cm, yshift=-3.5cm]
 \node[rv, rectangle]  (1)              {$1$};
 \node[rv, right of=1] (2) {$2$};
 \node[rv, right of=2, rectangle] (3) {$3$};
 \node[rv, right of=3] (4) {$4$};
 \draw[->, very thick, color=blue] (1) -- (2);
 \draw[->, very thick, color=blue] (3) -- (4);
 \draw[<->, very thick, color=red] (2.45) .. controls +(1,1) and +(-1,1) .. (4.135);
  \node[below of=2, xshift=10mm, yshift=10mm] {(c)};
\end{scope}
 \end{tikzpicture}
 \caption{Subgraphs corresponding to factorization of the graph in Figure \ref{fig:verma:mdag} into districts.  Parent nodes of the district are drawn as squares.}
 \label{fig:subs}
 \end{center}
\end{figure}
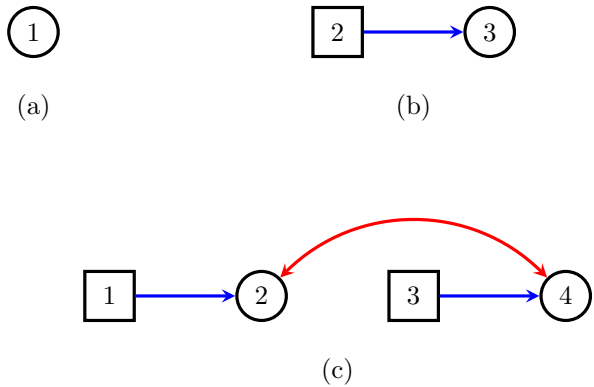

$\G[C]$ is therefore the subgraph induced over $C$, together with
parents of $C$ and edges directed towards $C$.  Any edges (whether
directed or bidirected) between the newly fixed vertices are removed.

For the graph in Figure \ref{fig:verma:mdag} the subgraphs
$\G[\{1\}]$, $\G[\{3\}]$ and $\G[\{2,4\}]$ are shown in Figures
\ref{fig:subs}(a), (b) and (c) respectively.  Note in particular that
the edge $2 \rightarrow 3$ is not in the subgraph $\G[\{2,4\}]$.

\begin{dfn}
  Let $\G$ be an mDAG with random vertices $V$.  For an arbitrary set
  $C \subseteq V$, define $\sterile_\G(C) \equiv C \setminus
  \pa_\G(C)$.  In words $\sterile_\G(C)$ is the subset of $C$ whose
  elements have no children in $C$.  We say a set $C$ is
  \emph{sterile} if $C = \sterile_\G(C)$.
\end{dfn}

\begin{prop} \label{prop:distfact} 
  Let $\G$ be an mDAG with districts $D_1, \ldots, D_k$.  A
  probability kernel $p$ is in the marginal 
  model for $\G$ if and only if
\begin{align*}
p(x_V \,|\, x_W) = \prod_{i=1}^k g_i(x_{D_i} \,|\, x_{\pa(D_i) \setminus D_i}),
\end{align*}
where each $g_i$ is a probability kernel in the marginal 
model for $\G[D_i]$.

In addition, $p$ is in the marginal model for $\G$ only if for every $v$ such
that $\ch_\G(v) = \emptyset$,
the marginal distribution
\begin{align*}
p(x_{V \setminus v} \,|\, x_W) = \sum_{x_v} p(x_V \,|\, x_W)
\end{align*}
is in the marginal model for $\G[V \setminus \{v\}]$.
\end{prop}

\begin{proof}
  Consider the factorization of the canonical DAG $\bar{\G}$.  The
  first result follows from grouping the factors according to
  districts and noting that there is no overlap in the variables being
  integrated out.  The second result follows from noting that if $v$
  has no children, the variable $x_v$ only appears in a single factor,
  and that factor is a conditional distribution that integrates to 1.
\end{proof}

It follows from this result that to characterize the marginal model 
we need only consider mDAGs containing a single district, since other 
models can always be reduced to combinations of such graphs. 

\subsection{Relationship between mDAGs and ADMGs}

Previous papers considering marginal and nested models for DAGs have used
\emph{acyclic directed mixed graphs}, which are the restriction of
mDAGs with random vertices so that each bidirected edge has size two
\citep{richardson:03, shpitser:12, evans:14, richardson:17}.

 \begin{figure}
 \begin{center}
 \begin{tikzpicture}[rv/.style={circle, draw, very thick, minimum size=6.5mm, inner sep=1mm}, node distance=20mm, >=stealth]
 \pgfsetarrows{latex-latex};
\begin{scope}
 \node[circle, minimum size=2mm, inner sep=0mm, fill=red] (0) {};
 \node[rv] (1) at (90:1.3) {$1$};
 \node[rv] (2) at (210:1.3) {$2$};
 \node[rv] (3) at (330:1.3) {$3$};
 \draw[->, very thick, color=red] (0) -- (1);
 \draw[->, very thick, color=red] (0) -- (2);
 \draw[->, very thick, color=red] (0) -- (3);
 \node[below of=1, yshift=-10mm]  {(a)};
\end{scope}
\begin{scope}[xshift=5cm]
 \node[rv] (1) at (90:1.3) {$1$};
 \node[rv] (2) at (210:1.3) {$2$};
 \node[rv] (3) at (330:1.3) {$3$};
 \draw[<->, very thick, color=red] (2) -- (1);
 \draw[<->, very thick, color=red] (3) -- (1);
 \draw[<->, very thick, color=red] (2) -- (3);
 \node[below of=1, yshift=-10mm]  {(b)};
\end{scope}
 \end{tikzpicture}
 \caption{(a) An mDAG on three vertices representing a saturated
   model; (b) the bidirected 3-cycle, the simplest non-geared mDAG.}
 \label{fig:cycles}
 \end{center}
\end{figure}
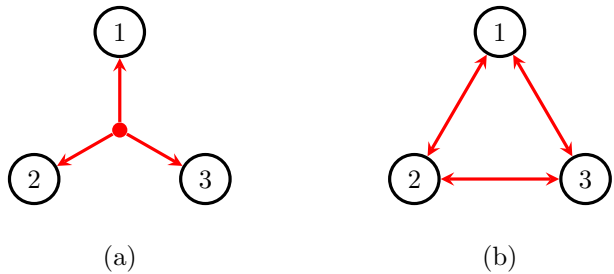

From the perspective of the nested Markov property this distinction is
unimportant: if we replace any bidirected simplicial complex with all
its subsets of size 2, we obtain a conditional ADMG that represents
the same model under the nested Markov property.  However, if we
consider the marginal model the models are not equal, as the restriction 
to pairwise independent latent parents will sometimes introduce 
additional inequality constraints.  The marginal model for the mDAG in
Figure \ref{fig:cycles}(b) is strictly smaller than the one for
\ref{fig:cycles}(a) \citep[][Proposition 2.13]{fritz:12}, for example.
See \citet{evans:mdag} for a more detailed discussion.

It follows from the results of this paper that there is no difference 
in equality constraints between graphs that differ only in this manner;
\emph{algebraically} the model
defined by having a single latent parent for several variables is the
same as having separate parents for each pair of vertices. The
mDAGs in Figure \ref{fig:cycles} both represent marginal models of full
dimension, for example.  
Hence, in terms of model dimension, nothing is lost by using ADMGs
instead of mDAGs.    

\section{Nested Markov Property} \label{sec:nested}

The nested Markov property imposes constraints 
on a joint distribution
that mimic those satisfied by the marginal model, including
conditional independences and the Verma constraint in Example
\ref{exm:one} \citep{richardson:17}.  It is defined in the following
recursive way, which is a modification of the algorithm of 
\citet{tian:02}.

\begin{dfn}[Nested Markov Property] \label{dfn:nmp}
  A kernel $p$ over $\X_V$ indexed by $\X_W$ 
  obeys the \emph{nested Markov property}
  for an mDAG $\G(V,W)$ if $V = \emptyset$, or both: 
\begin{enumerate}[1.]
\item $p$ factorizes over the districts $D_1, \ldots, D_l$ of $\G$:
\begin{align*}
p(x_V \,|\, x_W) = \prod_{i=1}^l g_i(x_{D_i} \,|\, x_{\pa(D_i) \setminus D_i})
\end{align*}
where each $g_i$ is a kernel 
which (if $l \geq 2$ or $W \setminus \pa_\G(V) \neq \emptyset$) obeys the nested Markov property
with respect to $\G[D_i]$; and
\item for each $v \in V$ such that $\ch_\G(v) = \emptyset$, the
  marginal kernel
\begin{align*}
p(x_{V \setminus v} \,|\, x_W) = \sum_{x_v} p(x_V \,|\, x_W)
\end{align*}
obeys the nested Markov property with respect to $\G[V\setminus \{v\}]$.
\end{enumerate} 

The set of kernels that obey the nested Markov property for $\G$
is the \emph{nested Markov model}, denoted by $\mathcal{N}(\G)$.
\end{dfn}

The condition that $l \geq 2$ or $W \setminus \pa_\G(V) \neq \emptyset$ in the first 
criterion of this definition is simply to prevent an infinite recursion 
of the definition: all the graphs invoked recursively have either 
fewer random vertices or fewer vertices overall 
than their predecessor in the recursion.  When we reach a graph
with a single random vertex $v$ such that all fixed vertices
are parents of $v$, then any kernel $p(x_v \mid x_{\pa(v)})$ satisfies
the nested Markov property.


\begin{exm} \label{exm:verma}
Consider again the mDAG in Figure \ref{fig:verma:mdag}.  
Applying criterion 1 to this graph implies that
\begin{align*}
p(x_1, x_2, x_3, x_4) = g_1(x_1) \cdot g_{24}(x_2, x_4 \,|\, x_1, x_3) \cdot g_3(x_3 \,|\, x_2)
\end{align*}
for some $g_1$, $g_3$ and $g_{24}$ obeying the nested Markov property
with respect to the mDAGs in Figures \ref{fig:subs}(a), (b) and (c)
respectively.  Applying the second criterion to $g_{24}$ and the now
childless vertex 2 (see Figure \ref{fig:subs}(c)) gives
\begin{align*}
\sum_{x_2} g_{24}(x_2, x_4 \,|\, x_1, x_3) = h(x_4 \,|\, x_3),
\end{align*}
for some function $h$ independent of $x_1$ (by a further application of the first criterion); 
this is precisely the Verma constraint. 

The marginal model implies additional conditions on joint
distributions because, although it satisfies the properties used to
define the nested model, these properties are not sufficient to
describe it.  In particular, for $p$ to be in the marginal model,
the kernel $g_{24}$ must satisfy Bell's inequalities 
\citep[see, for example,][Section 4.1]{steeg:11}.
\end{exm}

The nested Markov property is `sound' with respect to marginal models,
in the sense that all constraints represented by the former also hold
in the latter. 

\begin{thm}
For any mDAG $\G$ we have
$\M(\G) \subseteq \mathcal{N}(\G)$.
\end{thm}

\begin{proof}
  This follows from the fact that the nested Markov model is defined
  in terms of constraints which are proven in Proposition
  \ref{prop:distfact} to be satisfied by the marginal model.
\end{proof}

\subsection{Parameterizing Sets}

\begin{dfn}
  Let $\G$ be an mDAG.  A subset of random vertices $S \subseteq V$ is called
  \emph{intrinsic} if $S$ is a district in any graph which can be
  obtained by iteratively applying graphical operations of the form
  1 and 2 in Definition \ref{dfn:nmp} (i.e.\ taking the graph $\G[D]$
  for a district $D$, or $\G[V \setminus \{v\}]$ for a sterile vertex $v$).

  Given an intrinsic set, $S$, define $H = \sterile_\G(S)$ to be the
  \emph{recursive head}, and $T = \pa_\G(S)$ the \emph{tail},
  associated with $S$ (note that $H$ and $T$ are disjoint).  The
  collection of all recursive heads in $\G$ is denoted by
  $\mathcal{H}(\G)$.  There is a one-to-one 
correspondence between intrinsic sets and recursive heads 
\citep{evans:param}.
  Throughout we will use $H$ and $T$ to indicate recursive heads
and tails respectively, with the context making it clear which
intrinsic set is being referred to.  

Define 
\begin{align*}
\mathcal{A}(\G) \equiv \{H \cup A \,|\, H \in \mathcal{H}(\G), A \subseteq T\}
\end{align*}
to be the \emph{parameterizing sets} of $\G$.  This collection of sets is 
so-called because it (locally) describes the set of distributions (or kernels) 
contained in the nested and marginal models, as we prove in Section \ref{sec:main}. 

Conversely, non-empty sets not in $\mathcal{A}(\G)$ are called the 
\emph{constrained sets}, and locally describe the set of constraints 
imposed by the nested and marginal models.
\end{dfn}

\begin{exm}
  The mDAG in Figure \ref{fig:verma:mdag} has districts $\{1\}$,
  $\{3\}$ and $\{2,4\}$, so these are all intrinsic sets.  Further, in
  the subgraph $\G[\{2,4\}]$ the vertices 2 and 4 have no children, so
  we can marginalize either to see that respectively $\{4\}$ and
  $\{2\}$ are intrinsic sets.  The corresponding recursive heads and
  tails are then:
\begin{center}
\begin{tabular}{c||c|c|l}
$S$  & $H$   & $T$   & $\mathcal{A}$\\[1pt]
\hline
\{1\}     & \{1\}     &$\emptyset$ & $\{1\}$\\[3pt]
\{2\}     & \{2\}     &\{1\} & \{2\}, \: \{1,2\}\\[3pt]
\{3\}     & \{3\}     &\{2\} & \{3\}, \: \{2,3\}\\[3pt]
\{4\}     & \{4\}     &\{3\} & \{4\}, \: \{3,4\}\\[3pt]
\{2,4\}   & \{2,4\}   & \{1,3\} &  \{2,4\}, \: \{1,2,4\}, \: \{2,3,4\}, \: \{1,2,3,4\}\\
\end{tabular}.
%
\end{center}

Note that every non-empty subset of $V$ is represented in $\mathcal{A}$ except for
$\{1,3\}$, $\{1,2,3\}$, $\{1,4\}$ and $\{1,3,4\}$.  These are the `constrained
sets'; the first two correspond to the conditional independence,
$X_1 \indep X_3 \mid X_2$ in (\ref{eqn:ci}),
and the others to the Verma constraint (\ref{eqn:vc}).
\end{exm}

Intrinsic sets and recursive heads consist only of random
vertices, while tails may include both random and fixed vertices.  

\begin{prop} \label{prop:sterile} Let $C$ be a bidirected-connected
  set in an mDAG $\G$; then there exists an intrinsic set $S$ such
  that $C \subseteq S$ and $\sterile_\G(S) \subseteq \sterile_\G(C)$.
\end{prop}

\begin{proof}
  The district containing $C$ is intrinsic by definition, so there
  exists an intrinsic set containing $C$; let $S$ be a minimal
  intrinsic set (by inclusion) containing $C$.  By the definition of intrinsic
  sets $S$ is a district in some graph reached by iteratively applying the
  operations 1 and 2 to $\G$: applying operation 1 again gives the 
  graph $\G[S]$. 

  Suppose for contradiction that there exists $v \in \sterile_\G(S)
  \setminus \sterile_\G(C)$; then $v \notin C$, since otherwise some
  child of $v$ would be in $C$, and therefore in $S$.  In addition,
  $v$ is childless in the subgraph $\G[S]$, so we can remove $v$ under
  operation 2 of Definition \ref{dfn:nmp}.  In the resulting strictly
  smaller graph, $C$ is still contained within one district, say $S'$,
  since $C$ is bidirected-connected; in addition $S'$ is also
  intrinsic, so we have found a strictly smaller intrinsic set $S'
  \supseteq C$, and reached a contradiction.
\end{proof}

We use the $\triangle$ operator to denote the symmetric difference of
two sets: $A \triangle B \equiv (A \setminus B) \cup (B \setminus A)$. 
Given a collection of sets $A_i$, $i=1,\ldots,k$ indexed by a finite
set $I$, let
\begin{align*}
\bigtriangleup_{i=1}^k A_i \equiv A_1 \triangle A_2 \triangle \cdots \triangle A_k.
\end{align*}
denote the symmetric difference of all the $A_i$.  That is, it is the
set containing precisely those elements $a$ which appear in an odd
number of the sets $A_i$.

The following result gives a characterization of the parameterizing
sets in terms of symmetric differences which will be fundamental to
our proof of the main results in this paper.

\begin{lem} \label{lem:diffsets} 
  A set $A \in \mathcal{A}(\G)$ if and only if there exists a
  bidirected-connected set $C = \{v_1, \ldots, v_k\}$ in $\G$, and
  sets $A_i$, $i=1,\ldots,k$, satisfying
\begin{align*}
\{v_i\} \subseteq A_i \subseteq \{v_i\} \cup \pa_\G(v_i),
\end{align*}
such that
\begin{equation}
A = \bigtriangleup_{i=1}^k A_i = A_1 \triangle \cdots \triangle A_k.  \label{eqn:tri}
\end{equation}
\end{lem}

\begin{proof}
  Suppose that $A \in \mathcal{A}(\G)$; then $H \subseteq A \subseteq
  H \cup T$ for some head-tail pair $(H,T)$, with associated intrinsic
  set $S$.  Then let $C=S$, since intrinsic sets are by definition
  bidirected-connected.  Now consider sets $A'$ of the form (\ref{eqn:tri}); 
	start with $A_i = \{v_i\}$, so that $A' = S$, and we will adjust the sets
	$A_i$ to obtain $A'=A$.  
	We always have that $A'$ contains $H$, because each $v_i \in H$ appears in 
	$A_i$ and, by sterility of heads, in no other set $A_j$. 
	Each vertex $t \in T$ is 
  (by definition) the parent of some vertex $v_{j(t)} \in S$, so we 
  can either include or exclude it from $A'$ (as required) just by replacing 
  $A_{j(t)} = \{v_{j(t)}\}$ by $\{v_{j(t)}, t\}$.  Hence we just do this
  to include vertices in $A \setminus S \subseteq T$ and exclude vertices in 
  $S \setminus A \subseteq T$. 

  Conversely, suppose that $A$ is of the form (\ref{eqn:tri}) for some
  bidirected-connected set $C$; let $S$ be an intrinsic set satisfying
  the conditions of Proposition \ref{prop:sterile}, and $(H,T)$ be its
  associated head-tail pair.  Then the head $H = \sterile_\G(S)
  \subseteq \sterile_\G(C)$.  Each $v_i \in H \subseteq C$ appears in
  $A$, since $v_i \in A_j$ if and only if $i=j$.  Also $A \subseteq C
  \cup \pa_\G(C) \subseteq S \cup \pa_\G(S) = H \cup T$, so $A \in
  \mathcal{A}(\G)$.
\end{proof}

\subsection{Parameterization of the nested model}

The nested Markov model can be parameterized with parameters indexed
by head-tail sets \citep{evans:param}, and the parameterization defines
a smooth bijection between an open subset of a real vector space (i.e.\ the
parameter space) and the model (the set of probability distributions).  
This has some nice consequences that 
we now state \citep[for proofs see][]{evans:param}.

In particular, for a fixed state-space 
$\X_{VW}$ the set $\mathcal{N}(\G)$ is 
   a smooth manifold within the strictly positive probability simplex, and 
  has dimension
\begin{align*}
d(\G, \X_{VW}) \equiv \sum_{H \in \mathcal{H}(\G)} |\X_T| \prod_{h \in H} (|{\X}_h|-1).
\end{align*}
In the all-binary case this reduces to 
\begin{align*}
d(\G, \X_{VW}) \equiv \sum_{H \in \mathcal{H}(\G)} 2^{|T|}.
\end{align*}
Our main result will show that $\M(\G)$ always has the same dimension as 
$\mathcal{N}(\G)$.  Indeed, the parameterization of $\mathcal{N}(\G)$
will in principle also serve as a parameterization of $\M(\G)$, except that one would 
also have to restrict the parameter space in order to enforce the inequality 
constraints; of course, this is currently impractical 
since the inequality constraints are not generally known.


\section{Geared mDAGs} \label{sec:geared}

In this section we introduce a special class of mDAGs which 
we term `geared'.  For marginal models relating to such
graphs, the state-space of the hidden vertices can be restricted
without loss of generality, making proofs considerably easier.  
In Section \ref{sec:main} we prove
our main result first for geared graphs, and then extend the 
result to the general case.

\begin{dfn}
  Let $\G$ be an mDAG with bidirected simplicial complex
  $\mathcal{B}$.  We say that $\G$ is \emph{geared} if the maximal
  elements of $\mathcal{B}$ satisfy the running intersection property.
  That is, there is an ordering of the edges $B_1, \ldots, B_k$ such
  that for each $j > 1$, there exists $s(j) < j$ with
\begin{align*}
B_j \cap \bigcup_{i < j} B_i = B_j \cap B_{s(j)}.
\end{align*}
In other words, the vertices that are contained in both $B_j$ 
and any previous edge are all contained within one such edge $B_{s(j)}$. 

A particular ordering of the elements of $\mathcal{B}$ which satisfies
running intersection is called a \emph{gearing} of $\G$.\footnote{The 
term `geared' is chosen because a collection of bidirected edges
which satisfies running intersection may appear rather like `cogs' in
a set of gears: see Figure \ref{fig:mdag}.  The definition is
equivalent to the requirement that the simplicial complex
$\mathcal{B}$ is vertex decomposable \citep{provan:80}, and is also
closely related to the notion of decomposability in an undirected or
directed graph.  Indeed the term `decomposable' is used by
\citet{fox:14} to describe the same idea.  We avoid using this
terminology because of its existing meaning in connection with
undirected and directed graphical models: for example, 
ordinary DAGs are trivially
geared, but they may or may not be decomposable in the original sense
\citep{lau:96}.}
\end{dfn}

\begin{exm}
  The simplest non-geared mDAG is the bidirected 3-cycle, depicted in
  Figure \ref{fig:cycles}(b); there is no way to order
  the bidirected edge sets $\{1,2\}$, $\{2,3\}$, $\{1,3\}$ in a way
  which satisfies the running intersection property, since whichever
  edge is placed last in the ordering shares a different vertex with
  each of the two other edges. 
\end{exm}

The following fact about geared subgraphs of mDAGs will allow us to 
generalize our later results to graphs which are not geared.

\begin{lem} \label{lem:gearedsub} 
  Let $\G$ be an mDAG with parameterizing sets $\mathcal{A}(\G)$.  For
  any $A \in \mathcal{A}(\G)$ there exists
  a geared mDAG $\G' \subseteq \G$, such that
  $A \in \mathcal{A}(\G')$.
\end{lem}

\begin{proof}
  By Lemma \ref{lem:diffsets}, $A$ is of the form (\ref{eqn:tri}) for
  some bidirected-connected set $C$.  Let $\G'$ have the same
  vertices (random and fixed) and directed edges as $\G$, but be
  such that the set $C$ is \emph{singly 
  connected} by bidirected edges (i.e.\ the edges are all of size 2 
  and removing any of them will cause $C$ to be disconnected) chosen 
  to be a subgraph of $\G$.  
  Then $\G'$ is geared by standard properties of trees 
  and running intersection, and using Lemma \ref{lem:diffsets} again we 
  have $A \in \mathcal{A}(\G')$. 
\end{proof}

\subsection{Functional Models} 

The key property of geared graphical models is that we can find a
finite discrete latent variable model that is the same (over the 
observed variables) as the marginal model; that is, if the 
latent variables have a sufficiently large state-space then they
do not impose additional restrictions on the observed distribution. 
This is achieved by letting each observed variable be a deterministic 
function of its latent and observed parents. 
We illustrate this with an example.

\begin{exm}
  Consider the mDAG in Figure \ref{fig:iv}(a) representing the
  \emph{instrumental variables} model, used to
  model non-compliance in clinical trials; here, for example, $X_1$ represents a
  randomized treatment, $X_2$ the treatment actually taken, and $X_3$
  a patient's outcome or response, such as survival.  Suppose that each of these quantities is binary, 
  taking values in $\{0,1\}$.  Conceptually, it can be useful to posit the 
  existence of two different \emph{potential outcomes}  $X_3(0), X_3(1)$ 
  for the survival response, one for each level of the treatment; $X_3(0)$ is the 
  patient's outcome given that they choose not to take the treatment (so that $X_2=0$) 
  and $X_3(1)$ is their outcome given that they do ($X_2=1$).  For example, if $X_3(0) = 0$ and 
  $X_3(1) = 1$ then the patient survives if they take the treatment but dies
  if they do not.  This pair of values is known as a patient's \emph{response type}.  
  Of course, we can only ever observe one of these outcomes in 
  a given patient, the one corresponding to the observed value of $X_2$.  
  
  Similarly, we can conceive of two 
  versions of the treatment $X_2(0), X_2(1)$ depending upon the assigned 
  value of $X_1$, this pair being called the patient's \emph{compliance type}.
  For example, $X_2(0) = X_2(1) = 0$ means that the patient will
  not take the treatment, regardless of whether or not they are assigned to 
  the treatment group.  These concepts have proved fruitful in causal inference,
  as they enable discussion of whether treatments have effects at
  the level of individual patients, rather than just over the entire population
  on average \citep{neyman:23, rubin:74, richardson:11}. 

	Now, since the latent variable (say $U$) with children $\{2,3\}$ can take any
  value, we can---without loss of generality---assume that it includes the
  pair $(X_3(0), X_3(1))$, or equivalently a function $f_3 : \X_2 \rightarrow \X_3$
  that determines, given the observed $X_2$, which value $X_3$ will take.
  In this case $X_3$ is still a measurable function of its parents $U$ and 
  $X_2$.  Similarly we can assume $U$ includes a function $f_2 : \X_1 \rightarrow \X_2$
  that determines $X_2$ given an observed $X_1$.

 \begin{figure}
 \begin{center}
 \begin{tikzpicture}[rv/.style={circle, draw, very thick, minimum size=6.5mm, inner sep=1mm}, node distance=20mm, >=stealth]
 \pgfsetarrows{latex-latex};
\begin{scope}
 \node[rv] (1) {$1$};
 \node[rv, right of=1] (2) {$2$};
 \node[rv, right of=2] (3) {$3$};
 \draw[->, very thick, color=blue] (1) -- (2);
 \draw[->, very thick, color=blue] (2) -- (3);
 \draw[<->, very thick, color=red] (2) .. controls +(45:1cm) and +(135:1cm) .. (3);
 \node[below of=2, yshift=10mm]  {(a)};
\end{scope}
\begin{scope}[xshift=7cm]
 \node[rv] (1) {$1$};
 \node[rv, right of=1] (2) {$2$};
 \node[rv, right of=2] (3) {$3$};
 \node[rv, ellipse, inner sep=0mm, color=red, above of=2, xshift=10mm, yshift=-10mm] (U1) {$f_2,f_3$};
 \draw[->, very thick, color=red] (U1) -- (2);
 \draw[->, very thick, color=red] (U1) -- (3);
 \draw[->, very thick, color=blue] (1) -- (2);
 \draw[->, very thick, color=blue] (2) -- (3);
 \node[below of=2, yshift=10mm]  {(b)};
\end{scope}
 \end{tikzpicture}
 \caption{(a) An mDAG representing the instrumental variables model; (b) a DAG with functional latent variables equivalent to the potential outcomes model of instrumental variables.}
 \label{fig:iv}
 \end{center}
 \end{figure}
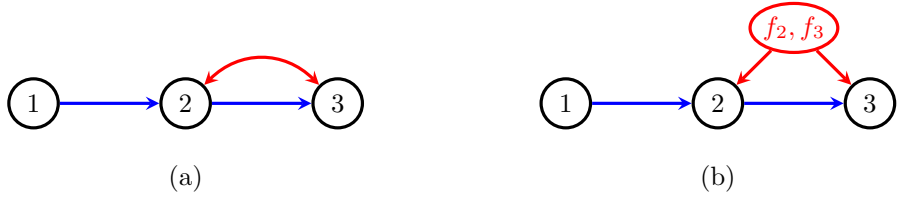

An observation for a particular patient can be obtained 
by drawing a random treatment assignment $X_1$, a 
random compliance type for the patient $f_2$, and a random 
response type $f_3$, and then evaluating 
$(X_1, X_2, X_3) = (X_1, f_2(X_1), f_3(f_2(X_1)))$.  
The key point is that one can place a distribution over 
$(X_1, f_2, f_3)$ and obtain a distribution over the 
observed variables $(X_1, X_2, X_3)$.  The only requirement
for the distribution to be Markov with respect to this 
particular graph is that $X_1 \indep \{f_2, f_3\}$,
as depicted in Figure \ref{fig:iv}(b).
\end{exm}

The functional construction outlined above is
mathematically equivalent to potential outcomes, and provides 
a model that is somewhat simpler to study than the general latent
variable model.  
In fact, any geared mDAG can be reduced to a latent variable model
in the way described above, something we now proceed to show.

\subsection{Remainder Sets}

Given a single-district, geared mDAG with at least one bidirected edge
and a gearing $B_1, \ldots, B_k$, define
\[
R_j \equiv B_j \setminus \bigcup_{i < j} B_i
\]
(taking $R_1 \equiv B_1$) to be the \emph{remainder set} associated with $B_j$. 
Remainder sets partition $V$, so
for a random vertex $v \in V$, define $r(v)$ to be the 
unique $j$ such that $v \in R_j$. 

Now say that an ordering $<$ on the vertices in $V$ \emph{respects the gearing} if for $v \in R_i$ and $w \in R_j$, we have $v < w$ whenever $i > j$; in other words, all the vertices in $R_k$ precede all those in $R_{k-1}$, etc; such an ordering always exists.  

For each $v \in V$ with $r(v) = j$, let
\begin{align*}
\pi(v) = \bigcup_{\substack{i > j \\ v \in B_i}} R_i;
\end{align*}
that is, the remainders associated with all bidirected edges which
contain $v$ and are later than $j$ in the ordering.  Then define a
collection of functions
\begin{align*}
\F_v \equiv \{f: \X_{\pa(v)} \times \F_{\pi(v)} \rightarrow \X_v\},
\end{align*}
where $\F_A = \times_{a \in A} \F_a$ and $\F_\emptyset = \X_\emptyset
= \{1\}$.  This is valid recursive definition, since all the vertices
in $\pi(v)$ precede $v$ in an ordering which respects the gearing.

\begin{exm} \label{exm:iv1}
The 
mDAG in Figure \ref{fig:iv}(a) has only one bidirected edge and therefore is trivially
geared with $R_1 = B_1 = \{2,3\}$.  This leads to the sets
\begin{align*}
\F_2 &= \{f_2 : \X_1 \rightarrow \X_2\}, &
\F_3 &= \{f_3 : \X_2 \rightarrow \X_3\},
\end{align*}
which are precisely the sets of functions for compliance type
and response type respectively.
\end{exm}

\begin{exm} \label{exm:gearing}
Consider the mDAG in Figure \ref{fig:mdag}, and order the bidirected edges as 
$B_1 = \{1,2\}$, $B_2 = \{2,3,4\}$ and $B_3 = \{3,4,5\}$,
giving respective remainder sets
$R_1 = \{1,2\}$, $R_2 = \{3,4\}$ and $R_3 = \{5\}$. 
The ordering $5 < 4 < 3 < 2 < 1$ of the random vertices respects the gearing, and we have
\begin{align*}
\pi(1) = \pi(5) = \emptyset, \qquad \pi(3) = \pi(4) = \{5\}, \qquad \pi(2) = \{3,4\}.
\end{align*}
In this case then
\begin{align*}
\F_5 &= \{f : \X_3 \rightarrow \X_5\} &
\F_4 &= \{f : \X_{2,3,6} \times \F_5 \rightarrow \X_4\}\\
\F_3 &= \{f : \X_1 \times \F_5 \rightarrow \X_3\} &
\F_2 &= \{f : \F_{3,4} \rightarrow \X_2\}\\
\F_1 &= \{f : \{1\} \rightarrow \X_1\} &&
\end{align*}
Alternatively, if we order the bidirected edges as $\{2,3,4\}$, $\{1,2\}$, $\{3,4,5\}$, then we could take $5 < 1 < 2 < 3 < 4$, and
\begin{align*}
\pi(1) = \pi(5) = \emptyset, \qquad \pi(3) = \pi(4) = \{5\}, \qquad \pi(2) = \{1\};
\end{align*}
this yields $\F_2 = \{f : \F_1 \rightarrow \X_2\}$, with other collections $\F_v$ remaining unchanged.  
\end{exm}

\subsection{Functional Models for Geared Graphs}

If a vertex $v$ is contained within exactly one bidirected edge, $B$,
then without loss of generality we can assume that the latent variable
corresponding to $B$ contains all the residual information about how
$X_v$ should behave given the values of its visible parents,
$X_{\pa(v)}$.  In other words, the latent variable associated with $B$
includes a (random) function $f_v : \X_{\pa(v)} \rightarrow \X_v$
which, once instantiated, `tells' $X_v = f_v(X_{\pa(v)})$ which value 
it should take for each value of its other parents, exactly as in 
Example \ref{exm:iv1}.\footnote{Equivalently, one could
take a deterministic function $f_v$ and introduce an `error term' 
$E_v$ so that $X_v = f_v(X_{\pa(v)}, E_v)$, as in the non-parametric
structural equation models of \citet{pearl:09}.}  All the randomness
of $X_v$ is collapsed into $f_v$ and $X_{\pa(v)}$.  

If $v$ is contained within two or more bidirected edges, say
$B_i$ and $B_j$, we might say that $B_i$ tells $X_v$ what value 
to take for every value of its visible parents and the other latent
variables.  However, it is not clear how to define such a function until
the state-space associated with the other latent parents (i.e. $B_j$) has
already been fixed.  The decomposable structure of geared graphs makes
it possible to iteratively fix state-spaces for latent variables
without loss of generality. 




To see this, 
suppose we have a single-district, geared mDAG $\G$ with remainder
sets $R_1, \ldots, R_k$, and form the canonical DAG $\bar{\G}$ by
replacing each bidirected edge $B_i$ in $\G$ with a new vertex $u_i$,
such that $\ch_{\bar{\G}}(u_i) = B_i$.  Compare, for example, the
structure of the graphs in Figures \ref{fig:mdag}(a) and (b).

Note that each vertex $v \in R_k$ has a single latent
parent $u_k$ in $\bar\G$.  Then, without loss of generality,
incorporate the function $f_v : \X_{\pa_\G(v)} \rightarrow \X_v$ into 
the latent variable $U_k$.  We `replace' $U_k$ with the 
collection of such functions $f_{R_k} \in \mathcal{F}_{R_k}$.  

Each vertex $v \in R_{k-1}$ has latent parent $u_{k-1}$ and possibly 
also $u_k$; but since the state-space of $U_k$ has been fixed 
as $\mathcal{F}_{R_k}$, we can define 
$f_v : \X_{\pa(v)} \times \mathcal{F}_{R_k} \rightarrow \X_v$ for 
those $v$ with latent parents $u_{k-1}$ and $u_k$, and just
$f_v : \X_{\pa(v)} \rightarrow \X_v$ otherwise.  These functions
$f_{R_{k-1}}$ can be integrated into $U_{k-1}$, and the process
repeated for $i=k-2,\ldots,1$.

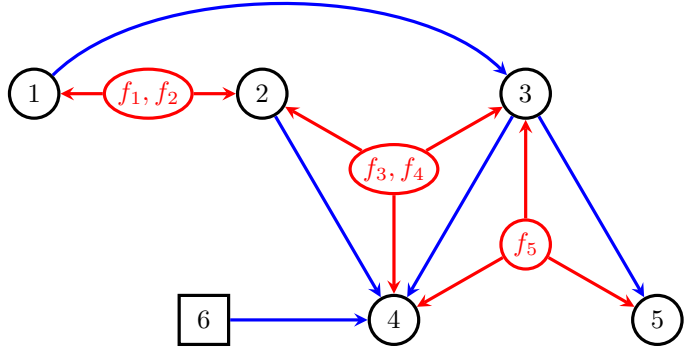
\begin{figure}
 \begin{center}
 \begin{tikzpicture}[rv/.style={circle, draw, very thick, minimum size=6.5mm, inner sep=1mm}, node distance=30mm, >=stealth]
 \pgfsetarrows{latex-latex};
 \node[rv, ellipse, inner sep=0mm, color=red] (U2) {$f_3,f_4$};
 \node[rv] (2) at (150:2) {$2$};
 \node[rv] (4) at (270:2) {$4$};
 \node[rv] (3) at (30:2) {$3$};
 \node[rv, left of=2] (1) {$1$};
 \node[rv, ellipse, inner sep=0mm, left of=2, xshift=15mm, color=red] (U1) {$f_1, f_2$};
 \node[rv, rectangle, left of=4, xshift=5mm] (5) {$6$};
 \node[rv, inner sep=0mm, color=red, yshift=-20mm] (U3) at (3) {$f_5$};
 \node[rv] (6) at ($(U3)+(330:2)$) {$5$};
 \draw[->, very thick, color=red] (U1) -- (1);
 \draw[->, very thick, color=red] (U1) -- (2);
 \draw[->, very thick, color=red] (U2) -- (2);
 \draw[->, very thick, color=red] (U2) -- (3);
 \draw[->, very thick, color=red] (U2) -- (4);
 \draw[->, very thick, color=red] (U3) -- (3);
 \draw[->, very thick, color=red] (U3) -- (4);
 \draw[->, very thick, color=red] (U3) -- (6);
 \draw[->, very thick, color=blue] (2) -- (4);
 \draw[->, very thick, color=blue] (5) -- (4);
 \draw[->, very thick, color=blue] (3) -- (4);
 \draw[->, very thick, color=blue] (3) -- (6);
 \draw[->, very thick, color=blue] (1) .. controls +(1.5,1.5) and +(-1.5,1.5) .. (3);
 \end{tikzpicture}
 \caption{A DAG with functional latent variables, associated with a gearing of the mDAG in Figure \ref{fig:mdag}(a).}
 \label{fig:canon}
 \end{center}
\end{figure}

\begin{figure}
 \begin{center}
 \begin{tikzpicture}[rv/.style={circle, draw, very thick, minimum size=6.5mm, inner sep=1mm}, node distance=30mm, >=stealth]
 \pgfsetarrows{latex-latex};
 \node[rv, ellipse, inner sep=0mm, color=red] (U2) {$f_3,f_4$};
 \node[rv] (2) at (150:2) {$2$};
 \node[rv] (4) at (270:2) {$4$};
 \node[rv] (3) at (30:2) {$3$};
 \node[rv, rectangle, left of=4, xshift=5mm] (5) {$6$};
 \node[rv, inner sep=0mm, color=red, yshift=-20mm] (U3) at (3) {$f_5$};
 \draw[->, very thick, color=red] (U2) -- (4);
 \draw[->, very thick, color=red] (U3) -- (4);
 \draw[->, very thick, color=blue] (2) -- (4);
 \draw[->, very thick, color=blue] (5) -- (4);
 \draw[->, very thick, color=blue] (3) -- (4);
 \end{tikzpicture}
 \caption{Subgraph of the DAG in Figure \ref{fig:canon} containing the vertex 4 and its parents.}
 \label{fig:func}
 \end{center}
\end{figure}
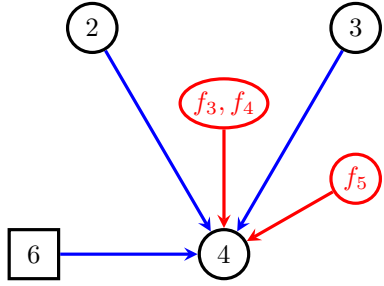

We end up with latent variables $U_i$ taking values in $\F_{R_i}$
for $i=1,\ldots,k$.
For example, with the first gearing given in Example \ref{exm:gearing} 
for the graph in Figure \ref{fig:mdag}(a), we would have
\begin{align*}
U_1 = (f_1, f_2), \qquad U_2 = (f_3, f_4), \qquad U_3 = (f_5).
\end{align*}
Associating each variable $U_i$ with the vertex $u_i$ leads to the DAG
in Figure \ref{fig:canon}. 
 Notice that, for each $v \in V$, the
function $f_v$ is contained within a parent variable of $v$.  In
addition, all the arguments of the function $f_v$ are also parents of
$v$. 
For example, take $v=4$, whose parents are drawn separately in Figure
\ref{fig:func}.  The function $f_4 \in \F_4$ is generated as part of the
latent variable $U_2 = (f_3, f_4)$, and the associated vertex $u_2$ is 
indeed a parent of $4$.  
In addition, $\F_4 = \{f : \X_{2,3,6} \times \F_5 \rightarrow
\X_4\}$, so the arguments of the function $f_4$, namely $X_2$, $X_3$,
$X_6$ and $f_5$, all correspond to vertices which are also parents of
$4$.  Thus, in
setting $X_4 = f_4(X_2, X_3, X_6, f_5)$ we ensure that $X_4$ is a well
defined function of its parent variables. 

In fact using this construction we can set
\begin{align*}
X_v = f_v(f_{\pi(v)}, X_{\pa(v)})
\end{align*}
for every $v \in V$, which is well defined because the directed part
of the original mDAG is acyclic.  The following result shows that the
resulting conditional distribution over $X_V$ given $X_W$ is in the
marginal model for the original mDAG.

\begin{thm} \label{thm:func} 
Let $\G$ be a geared mDAG, and $R_i, i=1,\ldots,k$ be the remainder sets
corresponding to some gearing of $\G$.  Suppose we generate
functions $f_v \in \F_v$ according to a distribution in which 
\begin{align*}
(f_v \,|\, v \in R_i) \indep (f_w \,|\, w \in V \setminus R_i),
\end{align*}
%
for each $i=1,\ldots,k$, and then define 
\begin{align*}
X_v = f_v(f_{\pi(v)}, X_{\pa(v)}), \qquad v \in V.
\end{align*}
Then the induced conditional distribution on $X_V$ given $X_W$
is in the marginal model for $\G$.
\end{thm}

\begin{proof}
For each bidirected edge $B_i$, define the random variable $U_i = (f_v
\,|\, v \in R_i)$.  The $U_i$s are represented by exogenous variables
on the DAG $\bar{\G}$, and the conditions given in the statement of
the theorem ensures they are all independent.  The structural equation
property for $\bar{\G}$ will therefore be satisfied if each $X_v$ is a
well defined function of its parents in the graph.

In other words, the three components $f_v$, $f_{\pi(v)}$ and
$X_{\pa(v)}$ must all be determined from random variables which are
parents of $v$ in $\bar{\G}$.  This holds for $X_{\pa(v)}$ by
definition.  Additionally $v \in R_i$ implies that $v \in B_i$, and
that therefore the variable $U_i \equiv (f_v : v \in R_i)$ is a parent variable of $X_v$.

Lastly suppose $w \in \pi(v)$; this happens if and only if $w,v \in
B_j$ for some $j > i$, in which case $w \in R_j$ for the minimal such
$j$ by the running intersection property of the gearing.  Then $f_w$
is contained in $U_j$, which is also a parent variable of $X_v$.
%
%
\end{proof}


In fact it is not hard to see that \emph{any} distribution in the marginal
model of a geared graph can be generated in the way described
in Theorem \ref{thm:func}.  Since each of these latent variables takes
values in a finite collection of functions, this means that the marginal
model of a geared graph is equivalent to a latent variable model in 
which all the random variables (latent and observed) are finite and
discrete.  It follows from this that marginal models for geared 
mDAGs are semi-algebraic sets by the Tarski-Seidenberg theorem 
\citep[][Chapter 2]{basu:06}. 

\begin{exm}
Consider the marginal model for the graph in Figure \ref{fig:verma:mdag}.
In this case the vertices 2 and 4 are each contained in only one 
bidirected edge, so without loss of generality this edge could be 
replaced in the canonical DAG (Figure \ref{fig:dag}) with a latent
variable taking values in $\F_2 \times \F_4$ where
\begin{align*}
\F_2 \equiv \{f : \X_1 \rightarrow \X_2\}, && \F_4 \equiv \{f : \X_3 \rightarrow \X_4\}.
\end{align*}
That is, the latent variable may be assumed to be $U = (f_2, f_4)$,
where $f_2$ and $f_4$ respectively assign values to $X_2$ and $X_4$ given 
particular values of $X_1$ and $X_3$. 
\end{exm}

For non-geared graphs such as that in Figure \ref{fig:4cycle}(a), 
there is no clear way to write the marginal model as a latent
variable model without possible loss of generality.  
It is therefore not possible for us to prove that marginal 
models corresponding to non-geared mDAGs are semi-algebraic;
however, we conjecture that they indeed are. 
Our results
will show that for a sufficiently large latent state-space the
dimension of the latent variable model becomes the same as that
of the nested model, but it is at least conceivable that there 
are non-polynomial inequality constraints on the marginal model
for non-geared graphs. 

\subsection{Generating Distributions for Geared mDAGs} 

Let $\G$ be a single-district, geared mDAG, with gearing given by
remainder sets $R_1, \ldots, R_k$; assign a probability distribution
$\rho_i$ to each collection of functions $U_i \equiv (f_v \,|\, v \in
R_i)$.  Suppose we draw values for variables $U_i = (f_v)_{v \in R_i}$
independently according to $\rho_i$, and use them to generate values for 
the observed
variables $X_V$ for each possible value of the fixed vertices $X_W$.  
The resulting (conditional) distribution over $X_V$ given $X_W$
is, by Theorem \ref{thm:func}, in the marginal model for $\G$.  

Let 
$\pi(R_i) \equiv \bigcup_{v \in R_i} \pi(v)$ and $f_A \equiv (f_{v} \,|\, v \in A)$.  
Define
\begin{align}
\P[\rho_k, \ldots, \rho_1](x_V \,|\, x_W) = \sum_{f_{R_k} \in \Phi_k(x_{VW})} \!\!\! \rho_k(f_{R_k}) \cdots \!\!\! \sum_{f_{R_1} \in \Phi_1(f_{\pi(R_1)}, x_{VW})} \!\!\! \rho_1(f_{R_1}), \label{eqn:sumdist}
\end{align}
where 
\begin{align}
\Phi_i(f_{\pi(R_i)}, x_{VW}) = \{f_{R_i} \,|\, f_{v}(x_{\pa(v)}, f_{\pi(v)}) = x_{v} \text{ for each } v \in R_i\}; \label{eqn:Phi}
\end{align}
that is, $\Phi_i(f_{\pi(R_i)}, x_{VW})$ is precisely the set of functions $f_{R_i}$
that, given the indicated values of parents variables, jointly evaluate
to $x_{R_i}$.  Hence (\ref{eqn:sumdist}) is a sum over all the combinations of
functions $f_V$ that, given the input $X_W = x_W$, recursively
evaluate to $x_V$. 

The function $\P[\cdot]$ takes distributions over the functions $f_V$ and 
returns a kernel over $\X_V$ indexed by $\X_W$.
For brevity we will generally denote this by
\begin{align*}
\P[\rho_k, \ldots, \rho_1] = \sum_{\Phi_k} \rho_k \cdots  \sum_{\Phi_1} \rho_1,
\end{align*}
with the dependence upon $x_{VW}$ left implicit.  It may be helpful to
think of this as an over-parameterized family of kernels for $X_V$ given
$X_W$, with parameters $\rho_1, \ldots, \rho_k$.

The mapping $\P[\cdot]$ is clearly smooth (infinitely 
differentiable), and its image defines the marginal model.  Hence 
we will be able to deduce various aspects of the model's geometry 
by studying $\P[\cdot]$ and its derivatives. 
Choosing $\rho_i(f_{R_i}) = 1$ for each $i$ (up to a
constant of proportionality which, for simplicity, we do not write
explicitly) induces the uniform distribution on $\X_{V}$ for each 
$x_W \in \X_W$; we denote this kernel by $p_0 \equiv \P[1, \ldots, 1]$.
Clearly $p_0$ is contained within $\M(\G)$ for any mDAG $\G$---as, 
in fact, is any distribution corresponding to all variables being 
independent.

\begin{exm} \label{exm:iv2}
For the instrumental variables model in Figure \ref{fig:iv} (if we consider
$X_1$ to be fixed), we have
\begin{align*}
\P[\rho](x_2, x_3 \mid x_1) &= \sum_{\Phi(x_{123})} \rho(f_2, f_3)
\end{align*}
where 
\begin{align*}
\Phi(x_{123}) = \{(f_2,f_3) : f_2(x_1) = x_2, f_3(x_2) = x_3 \}.
\end{align*}
\end{exm}

\begin{exm}
In the case of the mDAG in Figure \ref{fig:mdag}(a) we have three bidirected edges and remainder sets, and the gearing used in Figure \ref{fig:canon} gives 
\begin{align*}
\P[\rho_3, \rho_2, \rho_1]
  &= \sum_{\Phi_3} \rho_3(f_5)  \sum_{\Phi_2} \rho_2(f_3, f_4) \sum_{\Phi_1} \rho_1(f_1, f_2),
\end{align*}
where
\begin{align*}
\Phi_1 &= \{(f_1,f_2) \mid f_1 = x_1,\, f_2(f_3,f_4) = x_2\}\\
\Phi_2 &= \{(f_3,f_4) \mid f_3(x_1) = x_3, \, f_4(x_2, x_3, x_6, f_5) = x_4\}\\
\Phi_3 &= \{f_5 \mid f_5(x_3) = x_5\}.
\end{align*}
\end{exm}



\section{Main Results} \label{sec:main}

In this section we provide our main results, showing that the marginal
model $\M(\G)$ has the same dimension as the nested model.
This is done first for geared mDAGs, and the result is then extended to
general graphs.  For geared graphs, the marginal model is just
the image of the infinitely differentiable function $\P[\cdot]$ described in 
the previous section.  Such functions can be locally approximated at a 
particular point, say $p_0 = \P[1,\ldots,1]$, by the linear map given by the 
derivative of $\P[\cdot]$. 

This column space of this linear map (also called the \emph{pushforward} map) 
gives the linear space that approximates the model at $p_0$, also known as 
the tangent space.  We will show that the tangent space to the marginal model
at $p_0$ is equal to the tangent space of the nested model 
$\mathcal{N}(\G)$ at $p_0$.  To do this we take a basis of 
the tangent space of $\mathcal{N}(\G)$, and for every vector $\lambda$
in the basis we explicitly construct a vector $\delta$ such that the 
directional derivative of $p[\cdot]$ with
respect to $\delta$ is equal to $\lambda$.  This shows that each $\lambda$
is also contained in the tangent space of $\M(\G)$.
Since the marginal model is contained within the 
nested model, it will then follow from results in algebraic geometry that the
two models coincide in a neighbourhood of $p_0$.

For non-geared graphs we have do slightly more work, showing that
we can combine 
maps from different geared sub-graphs to obtain the same result. 

\subsection{Vector Spaces and Tangent Cones}

A probability kernel $p(x_V \,|\, x_W)$ can be thought
of equally as a vector with entries indexed by $\X_{VW}$, or a real function with 
domain $\X_{VW}$.  
The following decomposition of the
vector space $\reals^{|\X_{V}|}$ will prove useful. 

\begin{dfn}
For any $A \subseteq V$, let $\Lambda_A$ be the subspace of $\reals^{|\X_V|}$
consisting of vectors $p$ such that 
\begin{enumerate}[(i)]
\item $\sum_{y_a \in \X_a} p(y_a, x_{V \setminus a}) = 0$ for each $a \in A$ and $x_{V \setminus \{a\}} \in \X_{V \setminus \{a\}}$;
\item $p(x_V) = p(y_V)$ whenever $x_A = y_A$.
\end{enumerate}
\end{dfn}

In other words, considered as a function $p : \X_V \rightarrow
\reals$, the value of $p \in \Lambda_A$ only depends upon $x_A$, and
its sum over $x_a$ for $a \in A$ (keeping the other arguments fixed)
is 0.  In particular $\Lambda_{\emptyset}$ is the subspace spanned by
the vector of 1s.  The dimension of $\Lambda_A$ is 
$\prod_{a \in A} (|\X_a| - 1)$; in the case where all the variables are 
binary, each $\Lambda_A$ has dimension one and is the same as the space 
spanned by the corresponding column of a log-linear design matrix.

It is simple to check that the spaces $\Lambda_A$ are all orthogonal, 
and the real vector space
$\reals^{|\X_V|}$ can be decomposed as the direct sum
\begin{align*}
\reals^{|\X_V|} = \bigoplus_{A \subseteq V} \Lambda_A.
\end{align*}

\begin{dfn}
  Let $\mathfrak{A}$ be a subset of $\reals^k$ containing a point
  $\bs x$.  The \emph{tangent cone} of $\mathfrak{A}$ at $\bs x$ is
  the set of vectors of the form
\begin{align*}
\bs v = \lim_{n \rightarrow \infty} \eta_n^{-1}(\bs v_n - \bs x)
\end{align*}
where $\eta_n \rightarrow 0$ and each $\bs v_n \in \mathfrak{A}$. 
\end{dfn}


A tangent cone is a cone, but may or may not be a vector space,
depending upon whether the set $\mathfrak{A}$ is regular at $\bs x$. 
If $\mathfrak{A}$ is defined by the image of a differentiable bijective 
map then the tangent cone is a vector space, and the same as the
image of the pushforward map.  This is the case with the nested model  $\mathcal{N}(\G)$,
which has an explicit and smooth parameterization \citep{evans:param}.  Its tangent cone 
at the uniform distribution $p_0$ is the vector space 
\begin{align}
\TS_0^n \equiv \bigoplus_{A  \in \mathcal{A}(\G)} \Lambda_A, \label{eqn:vecspace}
\end{align}
where $\mathcal{A}(\G)$ are the parameterizing sets;
this can be deduced by looking directly at the 
parameterization.  

As noted in Section \ref{sec:geared}, any marginal model $\M(\G)$ 
also contains the uniform distribution
\begin{align*}
p_0(x_V \mid x_W) &\equiv |\X_V|^{-1}, && x_V \in \X_V, x_W \in \X_W,
\end{align*}
at which point all variables are totally independent.  
The tangent cone of the marginal model $\M(\G)$ at 
$p_0$ is also the vector space (\ref{eqn:vecspace}), which
forms the main result of this section.

\begin{thm} \label{thm:main}
The tangent cone of $\M(\G)$ at $p_0$, denoted $\TC_0$, is the vector space
\begin{align*}
\TC_0 = \TS_0^n \equiv \bigoplus_{A \in \mathcal{A}} \Lambda_A.
\end{align*}
\end{thm}

That $\TC_0 \subseteq \TS_0^n$ follows from the fact that
$\M(\G) \subseteq \mathcal{N}(\G)$.  
The proof of the reverse inclusion 
is delayed until the end of the section.


\subsection{Results for Geared Graphs}

\begin{dfn}
  Let $\lambda : \X_A \rightarrow \reals$; we say that $\lambda$ is
  $A$-\emph{degenerate} (or just degenerate) if for each $a \in A$,
  and $x_{A \setminus a} \in \X_{A \setminus a}$,
\begin{align*}
\sum_{y_a} \lambda(y_a, x_{A \setminus a}) = 0.
\end{align*}
\end{dfn}

It is not hard to see that the set of $A$-degenerate functions is isomorphic to
the vector space $\Lambda_A$; both formulations will be useful.

\begin{dfn}
Given a degenerate function $\vep_i : \F_{R_i} \rightarrow \reals$, define
\begin{align*}
D_i(\vep_i) = \lim_{\eta \downarrow 0} \eta^{-1} \left\{\P[1, \ldots, 1 + \eta\vep_i, \ldots, 1] - \P[1, \ldots, 1, \ldots, 1] \right\}, 
\end{align*}
so that $D_i(\vep_i)$ is a vector in $\reals^{|\X_{VW}|}$, the directional derivative
of the $i$th component of $\P[\cdot]$ with respect to $\vep_i$.  For
sufficiently small $\eta > 0$, the vector $1 + \eta \vep_i$ is non-negative and
therefore a valid distribution over $\F_{R_i}$ (up to the normalizing
constant); it follows that $D_i(\vep_i) \in \TC_0$, the tangent
cone of $\M(\G)$ at $p_0$.
\end{dfn}

Let 
$T_i = \{D_i(\vep_i) \,|\, \vep_i \text{ degenerate} \}$.
Since the function $\P[\cdot]$ is differentiable at $[1,\ldots,1]$ 
it follows that $T_i$ is a vector space, and also that the vector space 
$T_1 + \cdots + T_k$ 
is contained within the tangent cone of $\M$ at the uniform distribution. 
We will show that $T_1 + \cdots + T_k$ is in fact the same as 
(\ref{eqn:vecspace}). 

It will be useful to define the following collection of supersets of
$\Phi_i$, for $B \subseteq V$:
\begin{align}
\Phi_i^B(f_{\pi(R_i)}, x_{VW}) \equiv \{f_{R_i} \,|\, f_{v}(x_{\pa(v)}, f_{\pi(v)}) = x_{v} \text{ for each } v \in R_i \cap B\}. \label{eqn:Phi2}
\end{align}
That is, the collection of functions $f_{R_i}$ such that, given inputs
$f_{\pi(R_i)}$ and $x_{\pa(R_i) \setminus R_i}$, the values of 
$f_{B \cap R_i}$ jointly evaluate to $x_{B \cap R_i}$. 
Note that $\Phi_i^B = \Phi_i$ for any $B \supseteq R_i$.

\begin{lem} \label{lem:single} Let $C \subseteq R_i$, with
  $\sterile_\G(C) \subseteq A \subseteq C \cup \pa_\G(C)$ and $E
  \subseteq \pi(C)$.  Then for every degenerate function
\begin{align*}
\lambda : \X_A \times \F_E \rightarrow \reals,
\end{align*}
there exists a degenerate function $\delta : \F_C \rightarrow \reals$ such that
\begin{align*}
\sum_{f_{R_i} \in \Phi_i} \delta(f_{C}) = \lambda(x_A, f_E),
\end{align*}
where $\Phi_i$ is given by (\ref{eqn:Phi}).  In addition,
\begin{align*}
\sum_{f_{R_i} \in \Phi_i^B} \delta(f_{C}) = \left\{ 
\begin{array}{ll} |\X_{R_i \setminus B}| \lambda(x_A, f_E) & \text{if }C \subseteq B \\
0 & \text{otherwise}. \\\end{array} \right.
\end{align*}

\end{lem}

\begin{proof} 
See appendix, Section \ref{sec:singleproof}.
\end{proof}

Note that if we set $E = \emptyset$, the above result shows that
for any $\lambda \in \Lambda_A$ there exists a $\delta$ such that
\begin{align*}
  & \eta^{-1} \left\{\P[1, \ldots, 1+\eta \delta,\ldots, 1] - \P[1, \ldots, 1,\ldots, 1]\right\}\\
  &= \eta^{-1} \left\{ \sum_{\Phi_k} \cdots \sum_{\Phi_i} \eta \delta(f_C) \sum_{\Phi_{i-1}} \cdots \sum_{\Phi_{1}} 1 \right\}\\
  &= \lambda.
\end{align*}
Hence $\Lambda_A \leq T_i$ (i.e.\ $\Lambda_A$ is a subspace of $T_i$)
for any $A$ such that $\sterile_\G(C) \subseteq A \subseteq C \cup
\pa_\G(C)$ and $C \subseteq R_i$.  

The next result forms the backbone for proving Theorem \ref{thm:main}:
it extends Lemma \ref{lem:single} to sets $C$ that are not contained
within a single remainder set.

\begin{lem} \label{lem:multiple} Let $C$ be a bidirected-connected
  set, and for each $i$ define $C_i \equiv C \cap R_i$; let
  $I \equiv \{i \,|\, C_i \neq \emptyset\}$.  For
  $\sterile_\G(C_i) \subseteq A_i \subseteq C_i \cup \pa_\G(C_i)$, let
\begin{align*}
A = \bigtriangleup_{i \in I} A_i.
\end{align*}
Then $\Lambda_A \leq T_l$, where $l$ is the minimal element of $I$. 
\end{lem}

\begin{proof}
  By Lemma \ref{lem:pickatree} (see Appendix), there exists a rooted tree $\Pi$ with
  vertices $I$, such that $i \rightarrow j$ in $\Pi$ only if there
  exist $v_i \in R_i \cap C$ and $v_j \in R_j \cap C$ with $v_j \in
  \pi(v_i)$. 
  In particular $i \rightarrow j$ only if $C_j \subseteq \pi(v_i)$. 

  Let $l$ be the root node of $\Pi$, and for each $j \in \ch_{\Pi}(l)$
  denote by $\Pi_j$ the rooted tree with root $j$ formed only from the
  descendants of $j$.

  Let $\lambda_i : \X_{A_i} \rightarrow \reals$ be arbitrary
  $A_i$-degenerate functions for each $i \in I$.  Then starting with
  vertices which have no children (i.e.\ the leaves of the tree), and
  using Lemma \ref{lem:single}, recursively define $\delta_i$ for $i
  \in I$ as the degenerate function of $f_{C_i}$ such that
\begin{align*}
\sum_{\Phi_i} \delta_i(f_{C_i}) = \lambda_i(x_{A_i}) \prod_{j \in \ch_{\Pi}(i)} \delta_j(f_{C_j}),
\end{align*}
where the empty product is defined to be equal to 1.  Then
\begin{align*}
\sum_{\Phi_k} \cdots \sum_{\Phi_{l+1}} \sum_{\Phi_l} \delta_l(f_{C_l}) 
&= \lambda_l(x_{A_l}) \sum_{\Phi_k} \cdots \sum_{\Phi_{l+1}}  \prod_{j \in \ch_{\Pi}(l)} \delta_j(f_{C_j}).
\end{align*}
For each $i \in I$ an expression of the form $\sum_{\Phi_i}
\delta_i(f_{C_i})$ is only a function of 
$f_{C_\alpha}$ (and $x_{A_i}$) for $\alpha
\in \ch_\Pi(i)$ and $\Pi$ is a tree, so the sum factorizes
into components only involving the descendants of each $j \in
\ch_\Pi(l)$:
\begin{align}
\sum_{\Phi_k} \cdots \sum_{\Phi_1} \delta_l(f_{C_l}) 
&= \lambda_l(x_{A_l}) \prod_{j \in \ch_{\Pi}(l)} \sum_{\substack{\Phi_s \\ s \in \dec_{\Pi}(j)}}   \delta_j(f_{C_j}). \nonumber
\intertext{But then for each $j$ the factor represents a disjoint sub-tree $\Pi_j$ with root node $j$, so we can just iterate this process within each factor, and get}
& = \prod_{i \in I} \lambda_i(x_{A_i}). \label{eqn:prodlam}
\end{align}
It follows that any function of the form (\ref{eqn:prodlam})
lies in $T_l$.  Since $\Lambda_A$ is spanned by such functions by
Lemma \ref{lem:diff} (see Appendix \ref{sec:proofs}) it follows that
$\Lambda_A \leq T_l$.
\end{proof}

\begin{cor} \label{cor:tc}
For geared graphs $\G$, we have
\begin{align*}
\bigoplus_{A \in \mathcal{A}(\G)} \Lambda_A \leq T_1 + \cdots + T_k.
\end{align*}
\end{cor}

\begin{proof}
  Reformulating Lemma \ref{lem:diffsets} slightly, for any $A \in
  \mathcal{A}(\G)$ there exists a bidirected-connected set $C =
  \bigcup_i C_i = \bigcup_i \{v_{i1}, \ldots, v_{ik_i}\}$, where $C_i
  = C \cap R_i$ (we have changed nothing other than to label the
  vertices $v_{ij}$ by which remainder set they are contained in).
Then $A$ is of the form
\begin{align*}
A = \bigtriangleup_{i, j} A_i^j = \bigtriangleup_{i} \left(\bigtriangleup_{j} A_i^j\right)
\end{align*}
for some sets $A_i^j$ such that $\{v_{ij}\} \subseteq A_i^j \subseteq
\{v_{ij}\} \cup \pa_\G(v_{ij})$.

Applying Lemma \ref{lem:diffsets} in reverse to the
bidirected-connected set $C_i$ shows that $A_i \equiv \bigtriangleup_j
A_i^j$ is in $\mathcal{A}(\G)$, and therefore satisfies
$\sterile_\G(C_i) \subseteq A_i \subseteq C_i \cup \pa_\G(C_i)$.  Then
by Lemma \ref{lem:multiple} the space $\Lambda_A$ is contained in some
$T_i$, $i=1,\ldots,k$.
\end{proof}

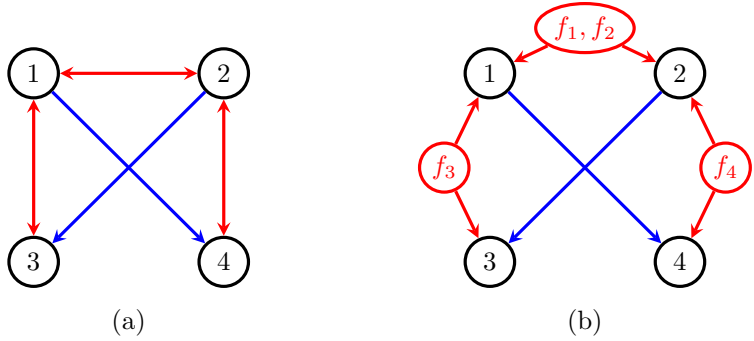
\begin{figure}
\begin{center}
\begin{tikzpicture}[rv/.style={circle, draw, very thick, minimum size=6.5mm, inner sep=0.5mm}, node distance=25mm, >=stealth] 
 \pgfsetarrows{latex-latex};
 \begin{scope}
 \node[rv] (1) {1};
 \node[rv, right of=1] (2) {2};
 \node[rv, below of=1] (3) {3};
 \node[rv, below of=2] (4) {4};
 \draw[->, very thick, color=blue] (1) -- (4);
 \draw[->, very thick, color=blue] (2) -- (3);
 \draw[<->, very thick, color=red] (1) -- (2);
 \draw[<->, very thick, color=red] (1) -- (3);
 \draw[<->, very thick, color=red] (4) -- (2);
 \node[xshift=12.5mm, yshift=-8mm] at (3) {(a)};
\end{scope}
\begin{scope}[xshift=6cm]
 \node[rv] (1) {1};
 \node[rv, right of=1] (2) {2};
 \node[rv, ellipse, xshift=12.5mm, yshift=6mm, color=red] (u1) at (1) {$f_1, f_2$};
 \node[rv, below of=1] (3) {3};
 \node[rv, yshift=-12.5mm, xshift=-6mm, color=red] (u2) at (1) {$f_3$};
 \node[rv, below of=2] (4) {4};
 \node[rv, yshift=-12.5mm, xshift=6mm, color=red] (u3) at (2) {$f_4$};
 \draw[->, very thick, color=blue] (1) -- (4);
 \draw[->, very thick, color=blue] (2) -- (3);
 \draw[->, very thick, color=red] (u1) -- (1);
 \draw[->, very thick, color=red] (u1) -- (2);
 \draw[->, very thick, color=red] (u2) -- (1);
 \draw[->, very thick, color=red] (u2) -- (3);
 \draw[->, very thick, color=red] (u3) -- (4);
 \draw[->, very thick, color=red] (u3) -- (2);
\node[xshift=12.5mm, yshift=-8mm] at (3) {(b)};
\end{scope}
\end{tikzpicture}
\end{center}
\caption{(a) an mDAG on 4 variables, and (b) a DAG with hidden variables corresponding to a gearing of the mDAG in (a).}
\label{fig:gadget}
\end{figure}

\begin{exm} \label{exm:iv3}
In the IV model from Figure \ref{fig:iv} (see Examples
\ref{exm:iv1} and \ref{exm:iv2}) has a saturated nested
model with the following parameterizing sets:
\begin{center}
\begin{tabular}{l | l | l}
head $H$   & tail $T$   & parametrizing sets $\mathcal{A}$\\[1pt]
\hline
\{1\}		&$\emptyset$	&\{1\}\\[1pt]
\{2\}		&$\{1\}$			&\{2\}, \{1,2\}\\[1pt]
\{3\}		&$\{1,2\}$		&\{3\}, \{1,3\}, \{2,3\}, \{1,2,3\}\\[1pt]
\end{tabular}.
\end{center}

Indeed, taking the functional parameterization suggested in 
Example \ref{exm:iv2}, one can see that altering the distribution
of the compliance functions $f_2$ will affect the distribution of
$X_2$ conditional on $X_1$, which is why $\Lambda_2$ and 
$\Lambda_{12}$ are contained in $\TC_0$.  
For example, to introduce a correlation between 
$X_1$ and $X_2$ whilst keeping the marginal distributions 
fixed, we can increase the proportion of `compliers'
(that is the people for whom $f_2(0) = 0, f_2(1) = 1$) and 
decrease the proportion of `defiers' ($f_2(0) = 1, f_2(1) = 0$).

Similarly, modifying the distribution of $f_3$ gives us 
$\Lambda_3$ and $\Lambda_{23}$.
Obtaining the directional derivatives in $\Lambda_{13}$ and 
$\Lambda_{123}$ requires modifying the distribution of $f_2,f_3$ 
jointly.  
\end{exm}

\begin{exm} \label{exm:gadget}
  Consider the single-district, geared mDAG in Figure
  \ref{fig:gadget}(a); the nested Markov model for this graph is
  saturated (i.e.\ has no constraints), 
  and thus its tangent space at $p_0$ is of full dimension.
   Correspondingly, one can check that 
   $\mathcal{A}(\G)$ consists of all non-empty
  subsets of $\{1,2,3,4\}$:


\begin{center}
\begin{tabular}{l | l | l}
head $H$   & tail $T$   & parametrizing sets $\mathcal{A}$\\[1pt]
\hline
\{1\}		&$\emptyset$	&\{1\}\\[1pt]
\{2\}		&$\emptyset$	&\{2\}\\[1pt]
\{1,2\}	&$\emptyset$	&\{1,2\}\\[1pt]
\{3\}		&$\{2\}$			&\{3\}, \{2,3\}\\[1pt]
\{1,3\}	&$\{2\}$			&\{1,3\}, \{1,2,3\}\\[1pt]
\{4\}		&$\{1\}$			&\{4\}, \{1,4\}\\[1pt]
\{2,4\}	&$\{1\}$			&\{2,4\}, \{1,2,4\}\\[1pt]
\{3,4\}	&\{1,2\}			&\{3,4\}, \{1,3,4\}, \{2,3,4\}, \{1,2,3,4\}
\end{tabular}.
%
\end{center}

Let us now see how our previous results apply to 
the marginal model in this case.  Consider the gearing 
\begin{align*}
B_1 &= \{1,2\} & B_2 &= \{1,3\} & B_3 &= \{2,4\}\\
R_1 &= \{1,2\} & R_2 &= \{3\} & R_3 &= \{4\}
\end{align*}
and ordering $4 < 3 < 1 < 2$ which respects this gearing.  This 
leads to the hidden variable model in Figure \ref{fig:gadget}(b); here
\begin{align*}
f_3 &: \X_2 \rightarrow \X_3 & f_4 &: \X_1 \rightarrow \X_4\\
f_1 &: \F_3 \rightarrow \X_1 & f_2 &: \F_4 \rightarrow \X_2.
\end{align*}
Applying Lemma \ref{lem:single} to each remainder set in turn tells us that 
\begin{align*}
\Lambda_1 + \Lambda_2 + \Lambda_{12} &\leq T_1,&
\Lambda_3 + \Lambda_{23} &\leq T_2, & 
\Lambda_4 + \Lambda_{14} &\leq T_3.
\end{align*}
%
We can apply Lemma \ref{lem:multiple} with the connected set 
$C = \{1,2,3,4\}$ to find that $\Lambda_A \leq T_1$, where $A$ 
is a set of the form $A = \{1,2\} \triangle A_2 \triangle A_3$ and 
\begin{align*}
\{3\} \subseteq A_2 \subseteq \{2,3\}, \qquad\qquad 
\{4\} \subseteq A_3 \subseteq \{1,4\};
\end{align*}
this gives us any $A \in \{ \{3,4\}, \{1,3,4\}, \{2,3,4\}, \{1,2,3,4\}\}$, and so
\[
\Lambda_{34} + \Lambda_{134} + \Lambda_{234} + \Lambda_{1234} \leq T_1.
\] 
Repeating with $C = \{1,2,3\}$ and $\{1,2,4\}$ respectively gives
$\Lambda_{13} + \Lambda_{123} \leq T_1$ and $\Lambda_{24} +
\Lambda_{124} \leq T_1$.

Thus for every non-empty $A \subseteq \{1,2,3,4\}$ there is some
$i\in\{1,2,3\}$ such that $\Lambda_A \leq T_i$, and therefore the
tangent cone of $\M(\G)$ at the uniform distribution is the
same as that of the saturated model on four variables.  In other words
the nested model and marginal model are both of full dimension.

\citet{evans:12} shows that the marginal model associated with
this graph induces some inequality constraints on the joint
distribution, and so the nested and marginal models are not
identical.
\end{exm}

\subsection{Extension to non-geared graphs}

Corollary \ref{cor:tc} put us in a position to prove Theorem
\ref{thm:main} for geared graphs; however it does not so far extend to
the general case, because we cannot fix the state-spaces of the latent
variables without a gearing.  In this section we will show that the
tangent cone of a general marginal model at the uniform
distribution is the vector space spanned by the tangent cones of its
geared subgraphs, and that therefore the problem can be reduced to
geared graphs.

\begin{figure}
 \begin{center}
 \begin{tikzpicture}[rv/.style={circle, draw, very thick, minimum size=6.5mm, inner sep=1mm}, node distance=20mm, >=stealth]
 \pgfsetarrows{latex-latex};
\begin{scope}
 \node[rv] (1) {$1$};
 \node[rv, right of=1] (2) {$2$};
 \node[rv, below of=2] (3) {$3$};
 \node[rv, below of=1] (4) {$4$};
 \draw[<->, very thick, color=red] (2) -- (1);
 \draw[<->, very thick, color=red] (3) -- (4);
 \draw[<->, very thick, color=red] (2) -- (3);
 \draw[<->, very thick, color=red] (4) -- (1);
 \node[below of=4, yshift=10mm, xshift=10mm]  {(a)};
\end{scope}
\begin{scope}[xshift=4cm]
 \node[rv] (1) {$1$};
 \node[rv, right of=1] (2) {$2$};
 \node[rv, below of=2] (3) {$3$};
 \node[rv, below of=1] (4) {$4$};
 \draw[<->, very thick, color=red] (2) -- (1);
 \draw[<->, very thick, color=red] (3) -- (4);
 \draw[<->, very thick, color=red] (2) -- (3);
 \node[below of=4, yshift=10mm, xshift=10mm]  {(b)};
\end{scope}
\begin{scope}[xshift=8cm]
 \node[rv] (1) {$1$};
 \node[rv, right of=1] (2) {$2$};
 \node[rv, below of=2] (3) {$3$};
 \node[rv, below of=1] (4) {$4$};
 \draw[<->, very thick, color=red] (2) -- (1);
 \draw[<->, very thick, color=red] (3) -- (4);
 \draw[<->, very thick, color=red] (4) -- (1);
 \node[below of=4, yshift=10mm, xshift=10mm]  {(c)};
\end{scope}
 \end{tikzpicture}
 \caption{(a) the bidirected 4-cycle, and (b), (c) two geared subgraphs.}
 \label{fig:4cycle}
 \end{center}
 \end{figure}
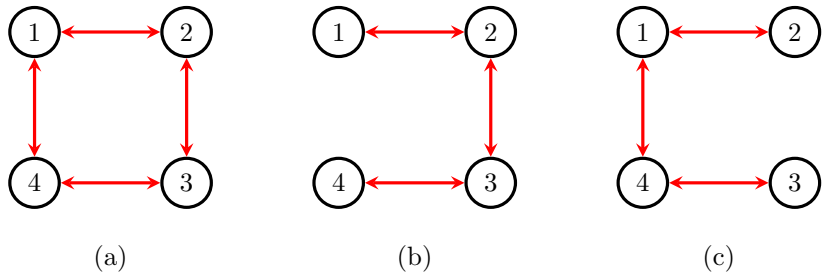

\begin{prop} \label{prop:mix} 
Let $\G$ be an arbitrary mDAG containing geared 
subgraphs $\G_1, \ldots, \G_k$.  Suppose that, for each subgraph and a 
suitable gearing $\Lambda_{A_i} \leq \TC_0(\G_i)$ as a consequence of 
the earlier results in this section.
Then $\Lambda_{A_1} + \cdots + \Lambda_{A_k} \leq \TC_0(\G)$.
\end{prop}

In other words, the tangent cone of $\G$ includes the vector space spanned
by all the tangent cones of the subgraphs.

\begin{proof}
  First consider the case $W = \emptyset$ and $k=2$, 
  from which the general result will follow similarly.

Let $p_1 \in \M(\G_1) \subseteq \M(\G)$ be formed by random functions 
$f_V$ according to a gearing of $\G_1$, and $p_2 \in \M(\G_2) \subseteq \M(\G)$ by
  random functions $\tilde{f}_V$ according to a gearing of $\G_2$.  
  Let $U_v$ be independent
  Bernoulli$(\frac{1}{2})$ variables, and define a new distribution by
  setting
\[
Z_v = U_v f_v(f_{\pi(v)}, Z_{\pa(v)}) + (1-U_v) \tilde{f}_v(\tilde{f}_{\tilde\pi(v)}, Z_{\pa(v)});
\]
i.e.\ we randomly (and independently of all other vertices) choose one
of the mechanisms $f_v$ or $\tilde{f}_v$ to generate $Z_v$.  
Denote $X_v \equiv f_v(f_{\pi(v)}, Z_{\pa(v)})$ and 
$Y_v \equiv \tilde{f}_v(\tilde{f}_{\pi(v)}, Z_{\pa(v)})$. 
Note that
although $f_v$ and $\tilde{f}_v$ are independent the values of
$X_v$ and $Y_v$ are not, since they share parent variables. 

Denote the resulting joint distribution of $Z_V$ by $p$.  We have
$p \in \M(\G)$ since we are still generating each variable
as a random function of its parents and some independent noise, which
clearly satisfies the structural equation property for $\bar{\G}$. 


Splitting into cases indexed by $B \equiv \{v : U_v = 1\}$, we have
\begin{align}
P(Z_V = z_V) &= \sum_{B \subseteq V} P(U_B = 1, U_{V \setminus B} = 0, X_B = z_B, Y_{V \setminus B} = z_{V \setminus B}) \nonumber\\
&= \frac{1}{2^{|V|}} \sum_{B \subseteq V} P(X_B = z_B, Y_{V \setminus B} = z_{V \setminus B }). \label{eqn:sumB}
\end{align}
It follows from the proof of Lemma \ref{lem:multiple} that if $A_1 \in
\mathcal{A}(\G_1)$ and $\lambda_{A_1} \in \Lambda_{A_1}$ then there exists a
degenerate $\delta(f_{C_i})$ such that
\begin{align*}
\sum_{\Phi_k} \cdots \sum_{\Phi_i} \delta(f_{C_i}) \cdots \sum_{\Phi_1} 1
 &= \lambda_{A_1}(x_{A_1})
\end{align*}
and from Lemma \ref{lem:single} that 
\begin{align}
\sum_{\Phi_k^B} \cdots \sum_{\Phi_i^B} \delta(f_{C_i}) \cdots \sum_{\Phi_1^B} 1
 &= \left\{ 
 \begin{array}{ll} 
|\X_{V \setminus B}| \lambda_{A_1}(x_{A_1}) & \text{if } C_i \subseteq B \\ 
0 & \text{otherwise.} 
\end{array} \right. \label{eqn:par}
\end{align}
Similarly, given $A_2 \in
\mathcal{A}(\G_2)$ and $\lambda_{A_2} \in \Lambda_{A_2}$
there exists a $\tilde{\delta}(\tilde{f}_{C_j})$ satisfying equivalent 
conditions over $\tilde{f}$s. 

Now since the functions used to generate $X_V$ and $Y_V$ are
independent,
\begin{align*}
\lefteqn{P(X_B = z_B, Y_{V \setminus B} = z_{V \setminus B})}\\ 
&= \left(\sum_{\Phi_k^B} \rho_k \cdots \sum_{\Phi_i^B} (\rho_i + \eta \delta ) \cdots \sum_{\Phi_1^B} \rho_1\right) \left(\sum_{\tilde\Phi_k^{V\setminus B}} \tilde\rho_k \cdots \sum_{\tilde\Phi_j^{V\setminus B}} (\tilde \rho_j + \eta \tilde{\delta}) \cdots \sum_{\tilde\Phi_1^{V\setminus B}} \tilde \rho_1\right)\\
&= \left(|\X_B|^{-1} + \eta c_1 \lambda_{A_1}\right) 
\left(|\X_{V \setminus B}|^{-1} + \eta c_2 \lambda_{A_2}\right)\\
&= |\X_V|^{-1} + \eta (c'_1 \lambda_{A_1} + c'_2 \lambda_{A_2}) + O(\eta^2).
\end{align*}
The first equality above follows just from consideration of which functions
we need to evaluate to which values in order to obtain $X_B = z_B$ and 
$Y_{V \setminus B} = z_{V \setminus B}$ (although this expression
apparently factorizes, note that both factors can depend upon all 
of $z_V$).  Note that possibly $c_i' = 0$, depending on which of the
conditions from (\ref{eqn:par}) are satisfied.
However, certainly $c_i' > 0$ for some subsets $B$, so 
plugging this back into (\ref{eqn:sumB}) we get
\begin{align*}
P(Z_V = z_V) &= 
|\X_V|^{-1} + \eta (c_1'' \lambda_{A_1} + c_2'' \lambda_{A_2}) + O(\eta^2),
\end{align*}
where $c_i'' > 0$.  Then by an appropriate choice of scaling for
each $\lambda_{A_i}$ we see that $\Lambda_{A_1} + \Lambda_{A_2} \leq
\TC_0(\G)$.  

For non-empty $W$, we can draw $Z_W = X_W = Y_W$ as a
uniform random variable, and then look at $X_V \,|\, X_W$; the proof
is otherwise the same.
\end{proof}

\begin{exm}
  The bidirected 4-cycle in Figure \ref{fig:4cycle}(a) is not geared,
  and therefore we cannot apply our earlier results to it directly.
  The nested model for this graph is equivalent to the model defined 
  by the constraints $X_1 \indep X_3$ and $X_2 \indep X_4$, and has 
  parameterizing sets
\begin{align*}
\mathcal{A}(\G) &= \{\{1\}, \; \{2\}, \; \{1,2\}, \; \{3\}, \; \{2,3\}, \; \{1,2,3\}, \\
								& \qquad		\; \{4\}, \; \{1,4\}, \; \{1,2,4\}, \; \{3,4\}, \; \{1,3,4\}, \; \{2,3,4\}, \; \{1,2,3,4\}\},
\end{align*}
which are also the bidirected-connected sets of vertices.
The two subgraphs in Figures \ref{fig:4cycle}(b) and (c), say $\G_1$ and $\G_2$, \emph{are} geared, however, and have parameterizing sets
\begin{align*}
\mathcal{A}(\G_1) &= \{\{1\}, \; \{2\}, \; \{1,2\}, \; \{3\}, \; \{2,3\}, \; \{1,2,3\}, \\
	&\qquad \; \{4\}, \; \{3,4\}, \; \{2,3,4\}, \; \{1,2,3,4\}\}\\
\mathcal{A}(\G_2) &= \{\{1\}, \; \{2\}, \; \{1,2\}, \; \{3\},\; \{4\}, \; \{1,4\},  \\
& \qquad \; \{1,2,4\}, \; \{3,4\}, \; \{1,3,4\}, \; \{1,2,3,4\}\};
\end{align*}
we have $\bigoplus_{A \in \mathcal{A}(\G_i)} \Lambda_A \leq
\TC_0(\G_i)$ for $i=1,2$ by Corollary \ref{cor:tc}.  Note that
$\mathcal{A}(\G_1) \cup \mathcal{A}(\G_2) = \mathcal{A}(\G)$, and
therefore by applying Proposition \ref{prop:mix} with these graphs, we
find that
\begin{align*}
\bigoplus_{A \in \mathcal{A}(\G)} \Lambda_A = \bigoplus_{A \in \mathcal{A}(\G_1)} \Lambda_A + \bigoplus_{A \in \mathcal{A}(\G_2)} \Lambda_A  \leq \TC_0(\G).
\end{align*}
It follows that the marginal model is also defined by the
independences $X_1 \indep X_3$ and $X_2 \indep X_4$, possibly with
some additional inequality constraints.
\end{exm}

We are now in a position to put together these ideas and prove
the main result for general mDAGs. 

\begin{proof}[Proof of Theorem \ref{thm:main}]
Suppose first that $\G$ is geared.  

$\P[1, \ldots, 1 + \eta\vep_i, \ldots, 1]$ obeys the nested Markov property for any degenerate function $\vep_i$ and $\eta$ sufficiently small that $1 + \eta\vep_i$ is positive; it follows that $T_i \leq \TC_0$ for each $i$, and that therefore using Corollary \ref{cor:tc},
\begin{align*}
\bigoplus_{A \in \mathcal{A}(\G)} \Lambda_A \leq T_1 + \cdots + T_k
\end{align*}
is also contained in $\TC_0$, by the differentiability of $\P[\cdot]$ at $(1, \ldots,1)$.  

Now for general $\G$, and each $A \in \mathcal{A}(\G)$, there exists a 
geared subgraph $\G'$ of $\G$ such that $\Lambda_A \leq \TC_0(\G')$ 
by Lemma \ref{lem:gearedsub}.  Then applying Proposition \ref{prop:mix}, 
we see that the space spanned by these subspaces is contained within the 
tangent cone for $\G$:
\begin{align*}
\bigoplus_{A \in \mathcal{A}(\G)} \Lambda_A \leq \TC_0(\G).
\end{align*}
If a distribution is in the marginal model then it
is also in the nested model, and therefore $\TC_0$ is
contained within the tangent space $\TS^n_0$ of $\mathcal{N}(\G)$ at $p_0$, which has dimension
\begin{align*}
\dim(\TS^n_0) &= \sum_{H \in \mathcal{H}(\G)} |\X_T| \prod_{h \in H} (|{\X}_h|-1)\\
&= \sum_{A \in \mathcal{A}(\G)} \dim(\Lambda_A);
\end{align*}
the second equality here follows from 
$\dim(\Lambda_A) = \prod_{h \in A} (|{\X}_h|-1)$ and
\begin{align*}
\sum_{H \subseteq A \subseteq H \cup T} \dim(\Lambda_A) 
= \sum_{H \subseteq A \subseteq H \cup T} \prod_{h \in A} (|{\X}_h|-1) 
= |\X_T| \prod_{h \in H} (|{\X}_h|-1).
\end{align*}
Then combining
\begin{align*}
\bigoplus_{A \in \mathcal{A}(\G)} \Lambda_A \leq \TC_0 \subseteq \TS_0^n
\end{align*}
with the dimension of $\TS_0^n$ gives the result.
\end{proof}

As a corollary of this result, within a neighbourhood of $p_0$
the models $\M(\G)$ and $\mathcal{N}(\G)$ are the same.  This is 
because $\mathcal{N}(\G)$ is parametrically defined via polynomials, 
and therefore its Zariski closure is an irreducible variety
\citep[see, e.g.][Proposition 4.5.5]{clo}.
For algebraic varieties $V_1, V_2$, if $V_1 \subseteq V_2$ 
and $V_2$ is irreducible, then either $V_1$ has a strictly 
smaller dimension than $V_2$, or they are identical.  The results
about the tangent space show that the Zariski closures of $\M$
and $\mathcal{N}$ have the same dimension, and therefore they are the 
same.  This means that, locally to $p_0$, the models themselves
are also the same.  

\section{Smoothness of the marginal model} \label{sec:smooth}

The results of Section \ref{sec:main}, together with the smoothness of
the nested model, allow us to show that for geared graphs, the
interior of the marginal model is a smooth manifold.

\begin{thm} \label{thm:manifold} For a geared graph $\G$ and
  state-space $\X_{VW}$, the relative interior of the marginal model $\M(\G)$ is
  a manifold of dimension $d(\G, \X_{VW})$, and its boundary is
  described by a finite number of semi-algebraic constraints.
\end{thm}

\begin{proof}
  The nested Markov model is parametrically defined (with a polynomial
  parameterization), and therefore its
  Zariski closure is an irreducible variety \citep[see,
  e.g.][Proposition 4.5.5]{clo}.  Furthermore \citet{evans:param} give a
  diffeomorphism between the set of strictly positive distributions
  obeying the nested Markov property, and an open parameter set.  It
  follows that $\mathcal{N}(\G)$ is a manifold on the interior of the
  simplex \citep[see, for example,][Appendix A]{kass:97}.

  As noted in the previous section, the marginal model for a geared
  graph is a semi-algebraic set.  Since the $\M(\G) \subseteq \mathcal{N}(\G)$
  and these two sets have the same Zariski closure, it follows that 
  $\M(\G)$ is defined from $\mathcal{N}(\G)$ 
  by a finite number of additional polynomial inequalities. 
  It further follows that it is also a manifold at any point these 
  inequality constraints are not active. 
\end{proof}

It follows from Theorem \ref{thm:manifold} that the interior of the
marginal model for a geared mDAG is a curved exponential family of
dimension $d(\G, \X_{VW})$, and that therefore the nice statistical
properties of these models can be applied.  For example, the maximum
likelihood estimator of a distribution within the model will be
asymptotically normal and unbiased, and the likelihood ratio statistic
for testing this model has an asymptotic $\chi^2$-distribution. 
  
  For non-geared mDAGs we cannot assume that the latent
variables are discrete without loss of generality, so it is
conceivable that these marginal models may be defined by
non-polynomial inequalities on the probabilities.  We conjecture,
however, that a result akin to Theorem \ref{thm:manifold} does hold
for general graphs. 
  
  For a point on the boundary defined by an active
inequality constraint, the asymptotic distribution of the likelihood ratio 
statistic may be much more
complicated \citep{drton:09a}; in general it is a mixture of $\chi^2$-distributions,
and this mixture will vary depending upon the unknown truth.  
%
%
A possible advantage of the nested model is that we can guarantee
that the true distribution does not lie on the boundary of $\mathcal{N}$ 
if the MLE consists of strictly positive probabilities, because the boundary
only consists of distributions with at least some zero probabilities; the same 
cannot be said for $\M$.  
This is depicted in Figure 
\ref{fig:diagram}, in which the MLE under the nested model ($\hat{p}_n)$ 
is in the interior of $\mathcal{N}$, but the MLE for the marginal model 
$\hat{p}_m$ lies on the boundary of $\M$.  

Inequality constraints are generally much more complicated than
equality constraints, and efforts to characterize them fully in DAGs with
latent variable models have been limited by computational challenges.  
\citet{evans:12}, generalizing a
result first given by \citet{pearl:95}, provides a graphical criterion
for obtaining some inequalities, but deriving a complete set of bounds 
may be an NP-hard problem \citep{steeg:11}.



\begin{figure}
\begin{center}

\begin{tikzpicture}
\coordinate (fl) at (-3,0);
\coordinate (fr) at (4,0);
\coordinate (pm) at (-0.5,0);
\coordinate (pn) at (1.5,0);
\coordinate (phat) at (1.5,1.5);
\node at (pm) [fill=black,circle,scale=0.5] {};
\node[below left] at (pm) {$\hat{p}_m$};
\node at (pn) [fill=black,circle,scale=0.5] {};
\node[below right] at (pn) {$\hat{p}_n$};
\node at (phat) {$\times$};
\node[above right] at (phat) {$\hat{p}$};
\draw[black,ultra thick] (fl) -- (pm);
\draw[black,ultra thick,dotted] (fl) -- (-3.5,0);
\draw[black] (fl) -- (fr);
\draw[black,dotted] (fr) -- (4.5,0);
\draw[black,dashed] (phat) -- (pn);
\draw[black,dashed] (phat) -- (pm);
  
  \draw[-latex,thick](-3,1.5) node[xshift=1mm,left,scale=1]{$\M$}
  to[out=0,in=110] (-1.5,0);
  \draw[-latex,thick](4.5,1.5) node[xshift=-1mm,right,scale=1]{$\mathcal{N}$}
  to[out=200,in=70] (3.5,0);
\end{tikzpicture}

\caption{Diagramatic representation of estimation with the marginal
model.  The thicker line represents the marginal model, and its thinner
extension the nested model.  
The unconstrainted MLE is shown as $\hat{p}$, and its projection 
to MLEs under the marginal and nested models as $\hat{p}_m$ and 
$\hat{p}_n$ respectively.  Note that $\hat{p}_m$ is on the boundary
of the model $\M$; if the true data generating distribution is on the
boundary this generally leads to irregular asymptotics.
}
\label{fig:diagram2}
\end{center}
\end{figure}
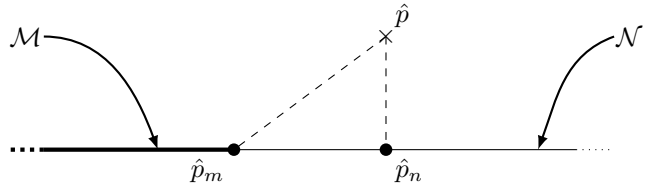

\subsection{Model Fitting}

In theory we can exactly fit the marginal model for a geared graph using a
latent variable model of the kind derived in Section \ref{sec:geared}.
In practice this model is massively over-parameterized and
unidentifiable, with the state-space of sets $\mathcal{F}_v$ being
potentially very large even for modest graphs; this will cause
problems for most standard fitting algorithms.  We can restrict the
state-space of the latent variables to something more managable to
obtain some latent variable model $\mathcal{L}(\G) \subseteq \M(\G)$;
in general the inclusion will be strict, as in the example in the 
introduction and as depicted in Figure \ref{fig:diagram}.  
However, for any graph $\G$---whether geared or not---and any latent
variable model we have
$\mathcal{L}(\G) \subseteq \M(\G) \subseteq \mathcal{N}(\G)$.  Fitting
the nested model by maximum likelihood (ML) is straightforward using
the algorithm in \citet{evans:10}, and a latent variable model can
be fitted using (for example) an EM algorithm.  A measure of 
goodness-of-fit for these two models can be used to bound the 
goodness-of-fit of the marginal model,
and thus potentially used to confirm or refute the marginal model.
Fitting the marginal model directly is likely to be extremely difficult
for general graphs: see the discussion in \citet{evans:mdag}.

\subsubsection*{Acknowledgements}

We thank Angelos Armen for a very close reading and substantial 
comments, as well as the associate editor and two anonymous
referees for their suggestions.

\bibliographystyle{abbrvnat}
\bibliography{mybib}

\appendix

\section{Degenerate Binary State-Space} \label{sec:deg}

Consider the model from Example \ref{exm:one} where we take
all five variables to be binary.  
The model can be written as
\begin{align*}
p(x_1, x_2, x_3, x_4) &= \sum_{x_0} p(x_0) \cdot p(x_1) \cdot p(x_2 \mid x_1, x_0) \cdot p(x_3 \mid x_2) \cdot p(x_4 \mid x_3, x_0)\\
&= p(x_1) \cdot p(x_3 \mid x_2) \cdot \underbrace{\sum_{x_0} p(x_0)  \cdot p(x_2 \mid x_1, x_0) \cdot p(x_4 \mid x_3, x_0)}_{p^*(x_2, x_4 \mid x_1, x_3)}.
\end{align*}
The dimension of the model is therefore the sum of the dimensions of 
these three factors (of which the first two are 1 and 2 respectively).
For the final factor, assuming $X_0$ is binary means it may be written
as
\begin{align}
p^*(x_2, x_4 \mid x_1, x_3) &= \sum_{x_0} p(x_0) \cdot p(x_2 \mid x_1, x_0) \cdot p(x_4 \mid x_3, x_0) \nonumber \\ 
&= \alpha \cdot q(x_2 \mid x_1) \cdot q(x_4 \mid x_3) + (1-\alpha) \cdot r(x_2 \mid x_1) \cdot r(x_4 \mid x_3) \label{eqn:lv}
\end{align}
for some distributions $q, r$.
This is a parametric definition of a variety over the
probabilities $p(x_2, x_4 \mid x_1, x_3)$, and using the 
computational algebra package {\sc Singular} \citep{singular}
we explicitly found the polynomial constraints that 
define it\footnote{For the code used and the resulting 
polynomial constraints, see 
\url{http://www.stats.ox.ac.uk/~evans/bell.html}.
As well as the usual marginal independence 
constraints, the latent variable model implies an additional cubic 
polynomial constraint on the observed conditional probabilities (for 
which we were unable to find a nice interpretation).}.  It 
turns out that the set of such probabilities that can
be written in the form (\ref{eqn:lv}) has dimension 7, 
and therefore the total dimension of the latent variable model
is $1 + 2 + 7=10$.

\section{Technical Proofs} \label{sec:proofs}






\subsection{Degenerate Functions}

We present a series of Lemmas which build up to showing that we can
construct degenerate functions from finite sums and products of
degenerate functions with simpler argument sets.

\begin{lem} \label{lem:rankone} Let $\lambda$ be a discrete $(A \cup
  B)$-degenerate function, for $A \cap B = \emptyset$.  Then $\lambda$
  can be written as a finite sum
\begin{align*}
\lambda = \sum_i \lambda_{A}^i \lambda_{B}^i
\end{align*}
of $A$-degenerate functions $\lambda_{A}^i$, and $B$-degenerate functions $\lambda_{B}^i$.
\end{lem}

\begin{proof}
Since a matrix can be written as a sum of rank one matrices, clearly we can find (not necessarily degenerate) functions such that the result holds.  
But now suppose that the $\lambda_{A}^i$ are not degenerate over $a \in A$, and consider
\begin{align*}
\lefteqn{\sum_i \left(\lambda_{A}^i(x_A) - \sum_{y_a} \lambda_{A}^i(x_{A \setminus a}, y_a)\right) \lambda_{B}^i(x_B)}\\
&= \sum_i \lambda_{A}^i(x_A) \lambda_{B}^i(x_B) - \sum_{y_a} \sum_i \lambda_{A}^i(x_{A \setminus a}, y_a) \lambda_{B}^i(x_B)\\
&= \lambda(x_A, x_B) - \sum_{y_a} \lambda(y_a, x_{A \setminus a}, x_B)\\
&= \lambda(x_A, x_B).
\end{align*}
Thus we can replace each $\lambda_{A}^i$ with the degenerate function 
\begin{align*}
\tilde{\lambda}_{A}^{i}(x_A) \equiv \left(\lambda_{A}^i(x_A) - \sum_{y_a} \lambda_{A}^i(x_{A \setminus a}, y_a) \right)
\end{align*}
and not affect the result.  By repeating the argument we can assume that each $\lambda_{A}^i$ is degenerate in every $a \in A$, and each $\lambda_{B}^i$ degenerate in every $b \in B$.
\end{proof}

\begin{lem} \label{lem:diff2}
Let $\lambda$ be a discrete $(A \triangle B)$-degenerate function.  Then $\lambda$ can be written as a finite sum
\begin{align*}
\lambda = \sum_j \lambda_{A}^j \lambda_{B}^j
\end{align*}
of $A$-degenerate functions $\lambda_{A}^j$, and $B$-degenerate functions $\lambda_{B}^j$.
\end{lem}

\begin{proof}
Let $A' = A \setminus B$ and $B' = B \setminus A$ and $D = A \cap B$, so that $A \triangle B = A' \cup B'$, $A = A' \cup D$ and $B = B' \cup D$; note that $A'$, $B'$ and $D$ are all disjoint.

For each $y_D \in \X_D$, define a degenerate function $\eta_D( \cdot; \, y_D) : \X_D \rightarrow \reals$ by
\begin{align*}
\eta_D(x_D; \, y_D) = \alpha^{-1} \prod_{d \in D} \left( |\X_d| \Ind_{\{x_d=y_d\}} - 1 \right).
\end{align*}
where $\alpha = \prod_{d \in D} |\X_d| \cdot (|\X_d|-1)$ and $\Ind$ denotes an indicator
function.  One can verify easily that 
\begin{align*}
\sum_{x_d \in \X_d} \eta_D(x_D; \, y_D) = 0
\end{align*}
for any $y_D$ and $x_{D\setminus d}$, and that
\begin{align*}
\sum_{y_D \in \X_D} \eta_D(x_D; \, y_D)^2 = 1;
\end{align*}
in particular the last expression is independent of $x_D$.

Now, let $\lambda$ be a discrete $(A \triangle B)$-degenerate function, and using Lemma \ref{lem:rankone} write it as
\begin{align*}
\lambda = \sum_{i=1}^j \lambda_{A'}^j \lambda_{B'}^j
\end{align*} 
where $\lambda_{A'}^j$ and $\lambda^j_{B'}$ are respectively $A'$ and $B'$ degenerate.  
Then for each $k \in \X_D$, define $\lambda_{A}^{jk} = \lambda_{A'}^j \eta_D( \cdot ; k)$ and $\lambda_{B}^{jk} = \lambda_{B'}^j \eta_D( \cdot ; k)$.  Clearly each of these is degenerate in $A = A' \cup D$ and $B=B'\cup D$ respectively.  Further, 
\begin{align*}
\sum_{i=1}^j \sum_{k \in \X_D} \lambda_{A}^{jk} \, \lambda_{B}^{jk} &= \sum_{i=1}^j \sum_{k \in \X_D} \lambda_{A'}^j \, \lambda_{B'}^j \, \eta_D(\cdot; \, k)^2\\
&= \sum_{i=1}^j \lambda_{A'}^j \, \lambda_{B'}^j \, \sum_{k \in \X_D} \eta_D(\cdot; \, k)^2\\
&= \sum_{i=1}^j \lambda_{A'}^j \, \lambda_{B'}^j\\
&= \lambda.
\end{align*}
\end{proof}

\begin{lem} \label{lem:diff} Let $\lambda : \X_A \rightarrow \reals$
  be an $A$-degenerate function, and let $A = \bigtriangleup_{i \in I}
  A_i$ for some finite collection of sets $\{A_i : i \in I\}$.  Then
  there exists a finite collection of $A_i$-degenerate functions
  $\lambda_i^j : \X_{A_i} \rightarrow \reals$ for $i \in I, j \in J$,
  such that
\begin{align*}
\lambda = \sum_{j \in J} \prod_{i \in I} \lambda_i^j.
\end{align*}
\end{lem} 

\begin{proof}
This just follows from repeatedly applying Lemma \ref{lem:diff2}.
\end{proof}

\subsection{Proof of Lemma \ref{lem:single}} \label{sec:singleproof}

\begin{lem} \label{lem:eval}
Let $\X$ and $\mathcal{Y}$ be finite sets, define 
$\F = \{f : \X \rightarrow \mathcal{Y}\}$, 
and take $\lambda : \mathcal{Y} \rightarrow \reals$.  
Then for any $A \subseteq \mathcal{Y}$ and $x \in \mathcal{X}$,
\begin{align*}
\sum_{\substack{f\in \F \\ f(x) \in A}} \lambda(f(x)) = |\mathcal{Y}|^{|\X|-1} \sum_{y \in A} \lambda(y),
\end{align*}
and if $x_1 \neq x_2$, 
\begin{align*}
\sum_{\substack{f\in \F \\ f(x_1) \in A}} \lambda(f(x_2)) = |A||\mathcal{Y}|^{|\X|-2} \cdot \sum_{y \in \mathcal{Y}} \lambda(y).
\end{align*}
\end{lem}

In particular note that if $\lambda$ is degenerate, the last expression is zero.

\begin{proof}
Clearly if $A = \mathcal{Y}$, then 
\begin{align*}
\sum_{\substack{f\in \F \\ f(x) \in \mathcal{Y}}} \lambda(f(x)) &= \sum_{f \in \F} \lambda(f(x))\\
&= |\mathcal{Y}|^{|\X|-1} \sum_{y \in \mathcal{Y}} \lambda(y),
\end{align*}
since there are exactly $|\mathcal{Y}|^{|\X|-1}$ functions in $\F$ such that $f(x) = y$ for each $y \in \mathcal{Y}$.  The first result follows in general by applying the result for $A=\mathcal{Y}$ to
the function $\lambda'(y) = \lambda(y) \Ind_{\{y \in A\}}$. 

The second result follows by similar combinatorical methods.
\end{proof}

\begin{proof}[Proof of Lemma \ref{lem:single}]
  It is clear that we only need prove the result for $E = \emptyset$,
  since we can just incorporate $f_E$ as though they were observable
  parents of $C$, and the result is the same.  

First consider the case $C = \{v\}$; let $L = \pa_\G(v)$ and take any set $K \subseteq L$.  Let $\lambda : \X_v \times \X_K \rightarrow \reals$ be a degenerate function, and for each $f: \X_L \times \F_{\pi(x)} \rightarrow \X_v$, define 
\begin{align*} 
\delta(f) = \sum_{\substack{y_{L} \in \X_{L} \\ g_{\pi(v)} \in \F_{\pi(v)}}} \lambda(f(y_{L}, g_{\pi(v)}), y_{K}).
\end{align*} 
Then for fixed $x_v, x_L, f_{\pi(v)}$,
\begin{align*}
\sum_{\substack{f \in \F_v \\ f(x_{L}, f_{\pi(v)}) = x_v}} \delta(f) &= \sum_{\substack{f \in \F_v \\ f(x_{L}, f_{\pi(v)}) = x_v}} \sum_{\substack{y_{L} \in \X_{L} \\ g_{\pi(v)} \in \F_{\pi(v)}}} \lambda(f(y_{L}, g_{\pi(v)}), y_{K})\\
&= \sum_{\substack{y_{L} \in \X_{L} \\ g_{\pi(v)} \in \F_{\pi(v)}}} \sum_{\substack{f \in \F_v \\ f(x_{L}, f_{\pi(v)}) = x_v}} \lambda(f(y_{L}, g_{\pi(v)}), y_{K}).
\intertext{But since $\lambda$ is degenerate, the inner sum is zero unless both $x_L = y_L$ and $f_{\pi(v)} = g_{\pi(v)}$ by Lemma \ref{lem:eval}.  This leaves }
  &= \sum_{\substack{f \in \F_v \\ f(x_{L}, f_{\pi(v)}) = x_v}} \lambda(f(x_{L}, f_{\pi(v)}), x_{K})\\
  &= |\X_v|^{|\X_L||\F_{\pi(v)}|-1} \cdot \lambda(x_v, x_{K})
\end{align*}
again by Lemma \ref{lem:eval}, where the constant represents the 
number of distinct functions $f \in \F_v$ such that $f(x_L, f_{\pi(v)}) = x_v$.  
Hence the result holds for $C = \{v\}$.

Now consider a general $C \subseteq R_i$; we prove the result by
induction on the size of $C$.  Given any $\sterile_\G(C) \subseteq A
\subseteq C \cup \pa_\G(C)$, we first claim that we can write $A = A_1
\triangle A_2$ where $\sterile_\G(C_i) \subseteq A_i \subseteq C_i
\cup \pa_\G(C_i)$ for $i=1,2$ and disjoint non-empty $C_1,C_2$ with
$C_1 \cup C_2 = C$.

To see this pick $C_2 = \{w\}$, $C_1 = C \setminus \{w\}$ for some $w
\in \sterile_\G(C)$, and then set $A_1 = (A \cup \sterile_\G(C_1))
\cap (C_1 \cup \pa_\G(C_1))$ and $A_2 = A \setminus A_1$.  Clearly
$A_1$ satisfies the required conditions.  Since $w$ was
chosen to be sterile in $C$ we have $w \notin A_1$ and therefore $w \in A_2$; in
addition, the only elements of $A$ not contained in $A_1$ are those
which are neither in $C_1$ nor $\pa_\G(C_1)$; but since they are in $C
\cup \pa_\G(C)$, they must instead be in $\{w\} \cup \pa_\G(w)$. 
Hence the claim holds. 

Now first suppose that 
$\lambda = \lambda_1 \cdot \lambda_2$ for degenerate functions $\lambda_i: \X_{A_i} \rightarrow \reals$.
By the induction hypothesis, we can find degenerate $\delta_1, \delta_2$ such that 
\begin{align*}
\sum_{\substack{f_v(x, f) = x_{v} \\ v \in C_1}} \delta_{1}(f_{C_1}) = c_1 \cdot \lambda_{1}(x_{A_1})\\
\sum_{\substack{f_v(x, f) = x_{v} \\ v \in C_2}} \delta_{2}(f_{C_2}) = c_2 \cdot \lambda_{2}(x_{A_2}).
\end{align*}
(Here we have written $f_v(x, f)$ for $f_v(x_L, f_{\pi_v})$ to reduce notational clutter.)
Then letting $E = R_i \setminus C$, 
\begin{align*}
\sum_{\substack{f_v(x, f) = x_{v} \\ v \in R_i}} \delta_{1}(f_{C_1}) \cdot \delta_{2}(f_{C_2}) &= 
\sum_{\substack{f_v(x, f) = x_{v} \\ v \in E}}
\sum_{\substack{f_v(x, f) = x_{v} \\ v \in C_1}}
\sum_{\substack{f_v(x, f) = x_{v} \\ v \in C_2}} \delta_{1}(f_{C_1}) \cdot \delta_{2}(f_{C_2})\\
&= c_0 \sum_{\substack{f_v(x, f) = x_{v} \\ v \in C_1}} \sum_{\substack{f_v(x, f) = x_{v} \\ v \in C_2}} \delta_{1}(f_{C_1}) \cdot \delta_{2}(f_{C_2})\\ 
&= c_0 \sum_{\substack{f_{C_1} \in \F_{C_1} \\ f_{C_1}(\x) = x_{C_1}}} \delta_{1}(f_{C_1})
\sum_{\substack{f_{C_2} \in \F_{C_2} \\ f_{C_2}(\x) = x_{C_2}}} \delta_{2}(f_{C_2})\\
&= c_0 c_1 c_2 \cdot \lambda_{1}(x_{A_1}) \cdot \lambda_{2}(x_{A_2}).
\end{align*}
However a general degenerate function $\lambda : \X_A \rightarrow \reals$ can be written as a finite linear combination 
\begin{align*}
\lambda = \sum_j \lambda_1^j \cdot \lambda_2^j
\end{align*}
of degenerate functions $\lambda_i^j : \X_{A_i} \rightarrow \reals$, so the result follows by linearity of summations.

For the final part, note that if $v \in C \setminus B$, then the
summation over $\Phi_i^B$ will include every function $f_v \in
\mathcal{F}_v$.  Since $\delta$ is degenerate and a function of $f_v$,
the sum is 0.  On the other hand, if $v \in (R_i \cap B) \setminus
C$, then $\delta$ is not a function of $f_v$ and summing over all
$\mathcal{F}_v$ just involves $|\X_v|$ identical terms.
\end{proof}

\begin{lem} \label{lem:pickatree}
Let $\G$ be a single-district, geared mDAG, and $C$ a bidirected-connected set of vertices.  There exists a rooted tree $\Pi_C$ with vertex set
\begin{align*}
I_C = \{i \,|\, R_i \cap C \neq \emptyset\},
\end{align*}
and edges $i \rightarrow j$ only if there exist $v_j \in R_j \cap C$ and $v_i \in R_i \cap C$ such that $v_j \in \pi(v_i)$.  
\end{lem}

\begin{proof}[Proof of Lemma \ref{lem:pickatree}]
  First construct a directed graph $\Pi^*$ on $I_C$ in which $i
  \rightarrow j$ precisely when there exist $v_j \in R_j \cap C$ and
  $v_i \in R_i \cap C$ such that $v_j \in \pi(v_i)$.  
  Note that $v_j \in \pi(v_i)$ implies $r(v_j) > r(v_i)$, so $\Pi^*$ 
  is acyclic.

Let $j$ be the minimal element of $I_C$; we claim that for any other
$i \in I_C$, there is always a directed path in $\Pi_C$ from $j$ to
$i$.  To see this, note that since $C$ is bidirected-connected, there
is a bidirected path in $\G$ from some $v_j \in C \cap R_j$ to $v_i
\in C \cap R_i$; given such a path, $\rho$, trim it so that only the
end-points are in $C \cap R_j$ and $C \cap R_i$ respectively.

If $\rho$ is just $v_j \leftrightarrow v_i$, then we are done, since
$v_j \in \pi(v_i)$ by definition of $\pi$.  Otherwise, $\rho$ begins
$v_i \leftrightarrow v_k \leftrightarrow \cdots$ for some $v_k \in R_k
\cap C$, where $i > k > j$.  So we can apply an inductive argument to
find a path from $j$ to $k$ in $\Pi_C^*$, and the edge $v_i
\leftrightarrow v_k$ implies that $k \rightarrow i$ in $\Pi_C^*$.

Now, $\Pi^*_C$ is a connected DAG with a unique root node $j$, so we 
can simply take any singly connected subgraph $\Pi_C$ to fulfil the 
conditions of the lemma.
\end{proof}


\end{document}